\documentclass[11pt]{article}

\usepackage{
    amsmath,
    amssymb,
    amsthm,
    caption,
    fullpage,
    thmtools,
    thm-restate,
    tikz
}
\usepackage[colorlinks=true, allcolors=blue]{hyperref}
\usepackage[capitalize,noabbrev]{cleveref}

\usetikzlibrary{decorations.pathreplacing}
\usetikzlibrary{arrows.meta}

\newtheorem{theorem}{Theorem}[section]
\newtheorem{corollary}[theorem]{Corollary}
\newtheorem{lemma}[theorem]{Lemma}
\newtheorem{proposition}[theorem]{Proposition}

\theoremstyle{definition}
\newtheorem{definition}[theorem]{Definition}
\newtheorem{observation}[theorem]{Observation}
\newtheorem{problem}[theorem]{Problem}
\newtheorem{example}[theorem]{Example}

\newcommand{\ith}[1]{$#1^{\text{th}}$}
\newcommand{\Homclass}[4]{\Hom_{#1\to#2}(#3,#4)}
\newcommand{\homclass}[4]{\hom_{#1\to#2}(#3,#4)}
\newcommand{\trans}{\text{T}}

\renewcommand{\loop}{\circ}

\DeclareMathOperator{\diag}{diag}
\DeclareMathOperator{\Hom}{Hom}
\DeclareMathOperator{\ind}{ind}

\title{Hoffman-London graphs: When paths minimize\\ $H$-colorings among trees}

\author{
David Galvin\thanks{Galvin is in part supported by Simons Foundation Gift no. 854277 - Travel Support for Mathematicians.}\\
University of Notre Dame\\
{\tt dgalvin1@nd.edu}
\and
Phillip Marmorino\thanks{Marmorino is in part supported by NSF AGS-10002554.}\\
Purdue University\\
{\tt pmarmori@purdue.edu}
\and
Emily McMillon\thanks{McMillon is in part supported by NSF DMS-2303380.}\\
Rice University\\
{\tt em72@rice.edu}
\and
JD Nir\\
Oakland University\\
{\tt jdnir@oakland.edu}
\and
Amanda Redlich\\
University of Massachusetts Lowell\\
{\tt amanda\_redlich@uml.edu}
}

\begin{document}

\maketitle

\begin{abstract}
Given a graph $G$ and a target graph $H$, an \textit{$H$-coloring} of $G$ is an adjacency-preserving vertex map from $G$ to $H$. The number of $H$-colorings of $G$, $\hom(G,H)$, has been studied for many classes of $G$ and $H$. In particular, extremal questions of maximizing and minimizing $\hom(G,H)$ have been considered when $H$ is a clique or $G$ is a tree.

In this paper, we develop a new technique using automorphisms of $H$ to show that $\hom(T,H)$ is minimized by paths as $T$ varies over trees on a fixed number of vertices. We introduce the term \emph{Hoffman-London} to refer to graphs that are minimal in this sense. 
In particular, we define an \textit{automorphic similarity matrix} which is used to compute $\hom(T,H)$ and give matrix conditions under which $H$ is Hoffman-London.  

We then apply this technique to identify several families of graphs that are Hoffman-London, including loop threshold graphs and some with applications in statistical physics (e.g.~the Widom-Rowlinson model). By combining our approach with a few other observations, we fully characterize the minimizing trees for all graphs $H$ on three or fewer vertices.
\end{abstract}

\section{Introduction} \label{sec:introduction}

Graph coloring is a well-known problem in which vertices of a graph are assigned a color with the restriction that adjacent vertices receive different colors. In this paper, we consider the related notion of $H$-colorings of a graph $G$ in which vertices of $G$ are labeled by vertices of $H$ in such a way that adjacent vertices in $G$ must be ``colored'' by adjacent vertices of $H$. Formally, given a simple, loopless graph $G = (V(G), E(G))$ and a target graph $H = (V(H),E(H))$ without multi-edges, but possibly containing loops, an \emph{$H$-coloring} of $G$ is a map $f: V(G) \to V(H)$ that preserves adjacency: whenever $u \sim_G v$ we must have $f(u) \sim_H f(v)$. Note that when $H= K_q$, the complete graph on $q$ vertices, an $H$-coloring of $G$ is an assignment of the vertices of $G$ to one of $q$ values with the only restriction that adjacent vertices receive different labels---in other words, the proper $q$-colorings of $G$. In this sense, $H$-colorings are a generalization of proper vertex colorings, though, as we will see, $H$-colorings generalize several other common graph theoretic notions as well.

While we mainly focus on the framework of $H$-colorings, it should be noted that an $H$-coloring of $G$ is just a graph homomorphism from $G$ to $H$. Fittingly, we denote the set of all $H$-colorings of $G$ by $\Hom(G,H)$ and use $\hom(G,H)$ to count the number of such colorings. 

As mentioned, $H$-colorings encompass proper $q$-colorings by allowing $H$ to be a complete graph. By setting $H = H_{\ind}$, which is one looped vertex connected by an edge to an unlooped vertex, we see $H$-colorings of $G$ also encode the \emph{independent} (or \emph{stable}) sets of $G$: any collection of vertices mapped to the unlooped vertex contains no internal edges. Lov\'asz~\cite{Lovasz2012} investigated connections between $H$-colorings and graph limits, quasi-randomness, and property testing. In statistical physics, the language of $H$-coloring has been adopted to describe hard-constraint spin models; see e.g. \cite{BrightwellWinkler1999, BrightwellWinkler2002}.

This article contributes to a broader investigation into an extremal enumeration question: for a given family of graphs $\mathcal{G}$ and a target graph $H$, can we characterize those $G \in \mathcal{G}$ which maximize or minimize $\hom(G,H)$? The origins of this question are attributed to Birkhoff's work on the 4-color conjecture~\cite{Birkhoff1912, BirkhoffLewis1946} with more recent attention coming from questions of Wilf~\cite{Wilf1984} and Linial~\cite{Linial1986} concerning which graph on $n$ vertices and $m$ edges admits the most proper $q$-colorings, which is to say maximizes $\hom(G,K_q)$. For further results and conjectures, we direct the reader to the surveys~\cite{Cutler2012, Zhao2017}.

Here we consider the case where $\mathcal{G}$ is a family of trees; in particular we consider the family of graphs $\mathcal{G} = \mathcal{T}_n$, the set of all trees on $n$ vertices. As in many such questions, the path graph $P_n$ and star graph $S_n = K_{1,n-1}$ are natural candidates for extremal graphs. When $H = H_{\ind}$, Prodinger and Tichy~\cite{ProdingerTichy1982} proved that for each $n$, the number of independent sets among $n$ vertex trees is maximized by stars and minimized by paths, confirming that for each $T_n \in \mathcal{T}_n$,
\begin{equation} \label{inq:PT}
\hom(P_n, H_{\ind}) \le \hom(T_n, H_{\ind}) \le \hom(S_n, H_{\ind}).
\end{equation}
The Hoffman-London inequality (see~\cite{Hoffman1967,London1966,Sidorenko1985}) is equivalent to the statement that $\hom(P_n,H) \le \hom(S_n,H)$ for \emph{any} choice of $H$ and $n$. Sidorenko~\cite{Sidorenko1994} extended the Hoffman-London inequality to the following very general result:

\begin{theorem}[Sidorenko~\cite{Sidorenko1994}] \label{thm:siderenko}
Fix $H$ and $n \ge 1$. For any $T_n \in {\mathcal T}_n$,
\[
\hom(T_n,H) \le \hom(S_n,H).
\]
\end{theorem}
In other words, for any target graph $H$ and order $n$, not only $P_n$, but every tree $T_n$ admits at most as many $H$-colorings as $S_n$. Sidorenko's result entirely resolves the maximization question for trees (and, it can be shown, connected graphs).

Given the result of Prodinger and Tichy, one might ask whether paths \emph{minimize} the number of $H$-colorings; could one show a result analogous to Sidorenko's that $\hom(P_n,H) \le \hom(T_n,H)$ for all target graphs? Csikv\'{a}ri and Lin~\cite{CsikvariLin2014}, following earlier work of Leontovich~\cite{Leontovich1989}, dash such hopes by describing a target graph $H$ and a tree $E_7$ on seven vertices in which $\hom(E_7, H) < \hom(P_7, H)$. They raise the following natural problem~\cite[Problem 6.2]{CsikvariLin2014}:
\begin{problem} \label{prob:CL}
Characterize those $H$ for which $\hom(P_n,H) \le \hom(T_n,H)$ holds for all $n$ and all $T_n \in {\mathcal T}_n$.
\end{problem}
Inspired by~\cite{Hoffman1967, London1966}, we introduce the following definitions:
\begin{definition} \label{defn:HL}
We say a target graph $H$ is \emph{Hoffman-London} if for all $n \ge 1$,
\begin{equation} \label{eq:HL-defn}
\hom(P_n,H) = \min_{T_n \in \mathcal{T}_n} \hom(T_n,H)
\end{equation}
If it holds that $\hom(P_n, H) < \hom(T_n, H)$ for all $T_n \in \mathcal{T}_n \setminus \{P_n\}$ and all $n \ge 1$, we say that $H$ is \emph{strongly Hoffman-London}.
\end{definition}

We can thus restate \cref{prob:CL} as: Which target graphs $H$ are Hoffman-London? To make progress on this problem, we develop criteria for identifying Hoffman-London and strongly Hoffman-London target graphs based on concepts introduced in a companion paper~\cite{GalvinMcMillonNirRedlich2025}, which we restate here. Given a target graph $H$, possibly with loops, the \emph{orbit partition} of $H$ is the partition of $V(H)$ obtained by declaring a pair of vertices $u, v$ to be equivalent if there is an automorphism that maps $u$ to $v$. We say such vertices are \emph{automorphically similar}. Denote the equivalence classes of the orbit partition, which we call \emph{automorphic similarity classes}, by $H^1, \ldots, H^k$, where $k$ is the number of automorphism orbits of $V(H)$. We will show in \cref{lem:auto_sim_nbrs} that for each $1 \le i \le k$ and $1 \le j \le k$, there is a constant $m_{i,j}$ counting the number of neighbors any $v \in H^i$ has in $H^j$, independent of the specific choice of $v$. The \emph{automorphic similarity matrix} of $H$ (relative to the specific ordering of the automorphic similarity classes) is the $k \times k$ matrix $M = M(H)$ whose $(i,j)$-entry is $m_{i,j}$.

We also make use of the following matrix property, which is related to the idea of stochastic domination.
\begin{definition} \label{def:incr-columns}
A $k \times k$ matrix $M$ with $(i,j)$-entry $m_{i,j}$ has the {\it increasing columns property} if any terminal segment of columns is non-decreasing: for each $1 \le c \le k$ and each $1 \le i \le k-1$,
\begin{equation} \label{inq:inc_cols}
\sum_{j=c}^k m_{i,j} \le \sum_{j=c}^k m_{i+1,j}.
\end{equation}
\end{definition}

Our first contribution, captured by \cref{thm:ordering_framework}, is a test that can be applied to the automorphic similarity matrix of a graph $H$ which, when it applies, proves that $H$ is Hoffman-London. 
\begin{restatable}{theorem}{orderings} \label{thm:ordering_framework}
Let $H$ be a target graph, possibly with loops. Suppose there is an ordering of the automorphic similarity classes $H^1, \ldots, H^k$ of $H$ such that the automorphic similarity matrix $M$ has the increasing columns property. Then $H$ is Hoffman-London. 
\end{restatable}

In certain cases, the reasoning behind \cref{thm:ordering_framework} can be extended to show $H$ is strongly Hoffman-London. For a graph $G$ with designated vertex $v$, let $\Homclass{v}{i}{G}{H}$ denote the set of $H$-colorings of $G$ in which $v$ is mapped to some (arbitrary but specific) representative of $H^i$. We use $\homclass{v}{i}{G}{H}$ for the number of such $H$-colorings. Although $\Homclass{v}{i}{G}{H}$ clearly depends on the choice of representative, in \cref{lem:using_iso} we will show that $\homclass{v}{i}{G}{H}$ does not.

\begin{restatable}{corollary}{uniquenessalt} \label{cor:uniquenessalt}
Let $H$ be a target graph that satisfies the hypotheses of \cref{thm:ordering_framework}. Suppose that for each $t \ge 2$, there exist two classes, $H^{a(t)}$ and $H^{b(t)}$, such that there is at least one $H$-coloring of $P_t$ that sends one endvertex of $P_t$ to a vertex in $H^{a(t)}$ and the other endvertex to a vertex in $H^{b(t)}$. Suppose further that for all $s \ge 2$, it holds that $\homclass{v}{b(t)}{P_s}{H} > \homclass{v}{a(t)}{P_s}{H}$, where $v$ is one of the end vertices of $P_s$. Then $H$ is strongly Hoffman-London. 
\end{restatable}

In \cref{sec:ordering_applications} we apply \cref{thm:ordering_framework} to prove (and in some cases, reprove known results) that several families of target graphs are Hoffman-London. These examples include vertex transitive graphs to which looped dominating vertices are added (partially recovering a result of Engbers and the first author~\cite{EngbersGalvin2017}), paths (partially recovering a result of Csikv\'ari and Lin~\cite{CsikvariLin2015}) and fully looped paths, loop threshold graphs, blow-ups of fully looped stars, and families of graphs (including the Widom-Rowlinson model) originating from statistical physics. When possible, we also apply \cref{cor:uniquenessalt} to show that these graphs are strongly Hoffman-London.

\section{Preliminary observations} \label{sec:prelim_obsv}

In this section, we gather some observations about enumerating $H$-colorings of trees that will prove useful at various points throughout the paper. We also present a proof that complete bipartite graphs are Hoffman-London. 

We begin with two observations pertaining to target graphs with more than one component. Both results follow immediately from the definition of $H$-coloring.

\begin{observation} \label{obs:union_of_comp}
If $H$ has components $H_1, \ldots, H_k$, then for any connected graph $G$ we have
\[\hom(G,H) = \sum_{i=1}^k \hom(G,H_i).\]
In particular, if it holds that $\hom(P_n,H_i) \le \hom(T_n,H_i)$ for all $i$, $n$, and $T_n \in {\mathcal T}_n$, then it holds that $\hom(P_n,H) \le \hom(T_n,H)$ for all $n$ and $T_n \in {\mathcal T}_n$.
\end{observation}

That is, the set of graphs that are Hoffman-London is closed under taking disjoint unions. The second observation follows from the first, but is useful on its own when considering graphs with isolated vertices.

\begin{observation} \label{obs:isolated}
If $H$ has isolated vertices and $H'$ is obtained from $H$ by removing the isolated vertices, then for each $G$ without isolated vertices we have $\hom(G,H)=\hom(G,H')$. In particular, this means that for the purpose of minimizing $H$-colorings of trees on two or more vertices, we can ignore isolated vertices. 
\end{observation}

We now make an observation about regular target graphs. Note that here and throughout this paper, the degree of a vertex $v$ is the number of edges (including loops) that contain $v$---so loops contribute $1$ to the degree, not $2$, as is used elsewhere in the literature. 

\begin{observation} \label{obs:regular}
If $H$ is regular, then all trees on $n$ vertices admit the same number of $H$-colorings. Specifically, $\hom(T,H) = |V(H)|d^{n-1}$, where $d$ is the degree of each vertex in $H$.
\end{observation}

It follows that regular graphs are, trivially, Hoffman-London. To see the validity of \cref{obs:regular}, put an ordering $v_1, \ldots, v_n$ on the vertices of $T$ satisfying that for each $i \ge 2$, $v_i$ is adjacent to exactly one $v_j$, $j < i$. For example, start with an arbitrary vertex as $v_1$, then list all the vertices at distance $1$ from $v_1$, in some arbitrary order, then move on to all the vertices at distance $2$ from $v_1$, and so on. When $H$-coloring the vertices of $T$ in this order, there are $|V(H)|$ options for the color at $v_1$, then $d$ options for the color of each successive vertex. 

Next, we refer to the notion of the tensor product of two graphs. For graphs $H_1$ and $H_2$, the graph $H_1 \times H_2$ has vertex set $V(H_1) \times V(H_2)$, and pairs $(x_1,y_1)$, $(x_2,y_2)$ are adjacent in $H_1 \times H_2$ if and only if $x_1x_2 \in E(H_1)$ and $y_1y_2 \in E(H_2)$; see \cref{fig:tensor}. 
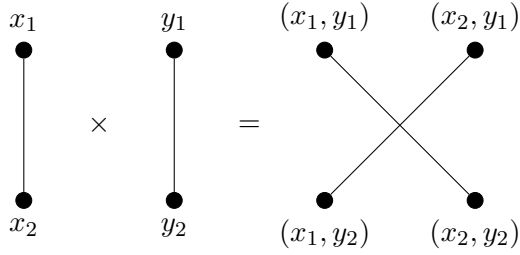
\begin{figure}[ht]
    \centering
    \begin{tikzpicture}
        \begin{scope}[xshift = -1cm]
            \node[circle,draw,fill,label={$x_1$},minimum size=6pt, inner sep=0pt] (a) at (0,2) {};
            \node[circle,draw,fill,label=below:{$x_2$},minimum size=6pt, inner sep=0pt] (b) at (0,0) {};
            \draw (a) -- (b);
        \end{scope}
        \node at (0,1) {$\times$};
        \begin{scope}[xshift = 1cm]
            \node[circle,draw,fill,label={$y_1$},minimum size=6pt, inner sep=0pt] (a) at (0,2) {};
            \node[circle,draw,fill,label=below:{$y_2$},minimum size=6pt, inner sep=0pt] (b) at (0,0) {};
            \draw (a) -- (b);
        \end{scope}
        \node at (2,1) {$=$};
        \begin{scope}[xshift = 3cm]
            \node[circle,draw,fill,label=above:{$(x_1,y_1)$},minimum size=6pt, inner sep=0pt] (a) at (0,2) {};
            \node[circle,draw,fill,label=below:{$(x_1,y_2)$},minimum size=6pt, inner sep=0pt] (b) at (0,0) {};
            \node[circle,draw,fill,label=above:{$(x_2,y_1)$},minimum size=6pt, inner sep=0pt] (c) at (2,2) {};
            \node[circle,draw,fill,label=below:{$(x_2,y_2)$},minimum size=6pt, inner sep=0pt] (d) at (2,0) {};
        
            \draw (a) -- (d);
            \draw (b) -- (c);
        \end{scope}
    \end{tikzpicture}
    \caption{The tensor product $K_2 \times K_2$.} \label{fig:tensor}
\end{figure}

There is a strong connection between tensor products and $H$-colorings, namely that for any graphs $G$, $H_1$, and $H_2$ we have 
\[
\hom(G,H_1 \times H_2) = \hom(G,H_1)\hom(G,H_2). 
\]
(see e.g.~\cite[Equation (5.30)]{Lovasz2012}). From this fact we draw the following two conclusions:
\begin{observation} \label{obs:cartesian_1}
If $H_1, H_2$ satisfy
$\hom(P_n,H_i) \le \hom(T_n,H_i)$ for each $i\in\{1,2\}$, $n \ge 1$, and $T_n \in {\mathcal T}_n$, then 
\[
\hom(P_n,H_1\times H_2) \le \hom(T_n,H_1 \times H_2)
\] 
for all $n$ and $T_n \in {\mathcal T}_n$.
\end{observation}
That is, the set of graphs that are Hoffman-London is closed under taking tensor products.

\begin{observation} \label{obs:cartesian_2}
Let $R$ be a regular graph of degree $d$ with at least one edge. Then for any $H$ and any $n \ge 1$, the trees $T_n$ in ${\mathcal T_n}$ that minimize $\hom(T_n,H)$ are the same as the trees $T_n$ in ${\mathcal T_n}$ that minimize $\hom(T_n,H \times R)$.
\end{observation}

\cref{obs:cartesian_2} follows from \cref{obs:cartesian_1} and from the fact that $\hom(T_n,R)=|V(R)|d^{n-1}$ (see \cref{obs:regular}) and is thus positive and independent of the specific choice of $T_n \in {\mathcal T}_n$. 

Finally, we present a result whose justification is more involved. A bipartition $V(G) = X \cup Y$ of a graph is called \emph{balanced} if the sizes are as equal as possible, i.e. $|X|= |Y|$ if $|V(G)|$ is even and $\bigl||X|-|Y|\bigr| = 1$ if $|V(G)|$ is odd.
\begin{theorem}\label{thm:complete_bipartite}
Each complete bipartite graph $H=K_{a,b}$ is Hoffman-London. When $a \ne b$, $T_n \in {\mathcal T}_n$ minimizes $\hom(T_n,H)$ if and only if $T_n$ has a balanced bipartition.
\end{theorem}

\begin{proof}
If $a=b$, then $H$ is a regular graph and \cref{obs:regular} tells us that every tree admits the same number of $H$-colorings. For the remainder of the proof we assume that $a \ne b$.

Any tree $T_n$ on $n$ vertices has a unique bipartition, say $V(T) = X \cup Y$, with $|X| = k$ and $|Y| = n-k$. Let $A$ be the part of $H$ containing $a$ vertices and $B$ be the part containing $b$ vertices. Let $\varphi \in \Hom(T_n,K_{a,b})$ and $x \in X$. Then, as $T_n$ is connected, each vertex in $X$ is adjacent to some vertex in $Y$ and each vertex in $Y$ is adjacent to some vertex in $X$, so if $\varphi(x) \in A$, we have $\varphi(X) \subseteq A$ and $\varphi(Y) \subseteq B$, and if $\varphi(x) \in B$, $\varphi(X) \subseteq B$ and $\varphi(Y) \subseteq A$. Furthermore, as $H$ is complete bipartite, the image of $x$ (and each other vertex in $X$) can be any vertex in the appropriate part, and similarly for the vertices of $Y$, so
\[ \hom(T_n,K_{a,b}) = a^kb^{n-k} + a^{n-k}b^k =: f(k).\]
We have
\[ 
f'(k) = \frac{\log(a)-\log(b)}{(ab)^k}(a^{2k}b^n-a^nb^{2k}).
\]
As $a \ne b$, the only zero of $f'$ occurs when $a^{n-2k}=b^{n-2k}$, which, for distinct positive integers, requires $n-2k = 0$, or $k = \frac{n}{2}$. Noting
\[ f''(k) = \frac{(\log(a)-\log(b))^2}{(ab)^k}(a^{2k}b^n+a^nb^{2k}) > 0, \]
we see $k = \frac{n}{2}$ is a minimum of $f$, and furthermore $f$ is decreasing as $k$ increases from $0$ to $\frac{n}{2}$. Therefore
\[
\min_{T \in \mathcal{T}_n} \hom(T_n,H) \ge 2a^{n/2}b^{n/2} 
\]
when $n$ is even. When $n$ is odd, since $k$ must be an integer,
\[
\min_{T \in \mathcal{T}_n} \hom(T_n,H) \ge a^{(n-1)/2}b^{(n+1)/2} + a^{(n+1)/2}b^{(n-1)/2}.
\]
In both cases, $P_n$ achieves the bound, as does any tree whose bipartition is balanced. If $T$ is a tree whose bipartition is not balanced, $f(k)$ is strictly greater than the minimum value.
\end{proof}

Csikv\'ari and Lin \cite{CsikvariLin2015} observed the case $a=1$ of \cref{thm:complete_bipartite}. The question of whether complete multipartite graphs with more than two parts are Hoffman-London remains open.

\section{Proofs of Theorem~\ref{thm:ordering_framework} and Corollary~\ref{cor:uniquenessalt}} \label{sec:main_theorem_proofs}

In this section, we prove \cref{thm:ordering_framework}, our primary condition for identifying Hoffman-London graphs $H$, and Corollary~\ref{cor:uniquenessalt}, our condition for identifying strongly Hoffman-London graphs. Recall from the introduction that if $H$ is a target graph, possibly with loops, and $u$ and $v$ are vertices in $H$, we say $u$ and $v$ are \emph{automorphically similar} if there is an automorphism of $H$ that sends $u$ to $v$. Automorphic similarity defines an equivalence relation on $V(H)$, and we call the equivalence classes of that relation the \emph{automorphic similarity classes} of $H$. The resulting partition of the vertices is sometimes called the \emph{orbit partition}. As discussed in~\cite[Section 9.3]{GodsilRoyle2001}, the orbit partition is an \emph{equitable partition}, which means that for any two (not necessarily distinct) automorphic similarly classes $A$ and $B$, there is a constant $c(A,B)$ such $|N(v) \cap B| = c(A,B)$ for each $v \in A$. In other words, vertices in the same automorphic similarity class have the same number of neighbors in each other class. For completeness, and because the proof is brief and illustrative, we include the following lemma (which was first presented in~\cite{GalvinMcMillonNirRedlich2025}).

\begin{lemma}\label{lem:auto_sim_nbrs}
Let $H$ be a graph, possibly with loops, and let $H^1, \ldots, H^k$ be the automorphic similarity classes of $H$. For any $1 \le i \le k$, any $u,v \in H^i$, and any $1 \le j \le k$,
\[ |N(u) \cap H^j| = |N(v) \cap H^j|.\]
\end{lemma}

\begin{proof}
Let $\varphi$ be an automorphism of $V(H)$ that sends $u$ to $v$. Then each $w \in N(u) \cap H^j$ is mapped to some $w'$. As $w \sim u$, we see $w' \sim v$, and as $w \in H^j$ and $\varphi(w) = w'$ we see $w$ and $w'$ are automorphically similar, so $w' \in H^j$ as well. We have shown $|N(u) \cap H^j| \le |N(v) \cap H^j|$. Repeating the analysis with $\varphi^{-1}$, we see $|N(u) \cap H^j| \ge |N(v) \cap H^j|$ and so equality holds.
\end{proof}

Using \cref{lem:auto_sim_nbrs}, given an ordering $H^1,\ldots, H^k$ of the automorphic similarity classes of $H$, we define $m_{i,j} = c(H^i, H^j)$ to be the number of neighbors that an arbitrary vertex in $H^i$ has in $H^j$. We then define an \emph{automorphic similarity matrix} of $H$ to be a $k \times k$ matrix $M = M(H)$ whose $(i,j)$-entry is $m_{i,j}$, where we note that different orderings of the automorphic similarity classes of $H$ may result in different automorphic similarity matrices.

As we will see in \cref{lem:tree_walk}, the matrix $M$, together with the sizes of the automorphic similarity classes of $H$, contains all of the information necessary to count $H$-colorings of trees. This could also be accomplished using the adjacency matrix of $H$, but for highly symmetric target graphs the automorphic similarity matrix is much smaller than the adjacency matrix and is thus easier to analyze. First, though, we need the following result, also presented in~\cite{GalvinMcMillonNirRedlich2025}. 

\begin{lemma} \label{lem:using_iso}
Let $H$ be a graph, possibly with loops, and let $G$ be an arbitrary graph with distinguished vertex $v$. Suppose $w$ and $w'$ are automorphically similar vertices in $H$. Let ${\mathcal G}$ be the set of $H$-colorings of $G$ in which $v$ is sent to $w$, and let ${\mathcal G}'$ be the set of $H$-colorings of $G$ in which $v$ is sent to $w'$. Then $|{\mathcal G}|=|{\mathcal G}'|$.
\end{lemma}

\begin{proof}
Let $\varphi$ be an automorphism of $H$ that sends $w$ to $w'$, which exists as $w$ and $w'$ are automorphically similar. Then the function mapping $f$ to $f \circ \varphi$ is a bijection mapping ${\mathcal G}$ to ${\mathcal G}'$. 
\end{proof}

We now present another lemma which we will require in the proof of \cref{thm:ordering_framework} but which we also feel is also of independent interest: a generalization to the setting of automorphic similarity classes and matrices of the tree-walk algorithm of Csikv\'ari and Lin \cite{CsikvariLin2014}. Recall that $\Homclass{v}{i}{T}{H}$ denotes the set of $H$-colorings of $T$ in which $v$ is mapped to some (arbitrary but specific) representative of $H^i$ and that we set $\homclass{v}{i}{T}{H}=|\Homclass{v}{i}{T}{H}|$.

\begin{lemma} \label{lem:tree_walk}
Let $H$ be a target graph with automorphic similarity classes $H^1, \ldots, H^k$ and associated automorphic similarity matrix $M$. Let $T$ be a tree with root $v$.  Let $\mathbf{a}(H)$ be the row vector whose \ith{i} entry is $|H^i|$ and define $\mathbf{h}(T,v)$ to be the column vector whose \ith{i} entry is $\homclass{v}{i}{T}{H}$.

Then $\hom(T,H)={\mathbf a}(H) \mathbf{h}(T,v)$, and $\mathbf{h}(T,v)$ can be explicitly computed from $T$, $\mathbf{a}(H)$, and $M$, via recursion on $T$.
\end{lemma}

\begin{proof}
To see that $\hom(T,H)={\mathbf a}(H) \mathbf{h}(T,v)$, partition $\Hom(T,H)$ according to the color given to $v$. There are $|H^i|$ classes in this partition in which $v$ is send to some vertex in $H^i$, and by \cref{lem:using_iso}, each of these classes has size $\homclass{v}{i}{T}{H}$. Hence the total count of $H$-colorings is
\[
\sum_{i=1}^k |H^i|\homclass{v}{i}{T}{H} = {\mathbf a}(H) \mathbf{h}(T,v).
\]

To compute $\mathbf{h}(T,v)$, we proceed recursively. When $T$ consists of a single vertex, necessarily $v$, we have that $\homclass{v}{i}{T}{H} = 1$ for any $i$, and so $\mathbf{h}(T,v)$ is the constant vector with all entries equal to $1$. 

When $T$ has more than one vertex, we consider separately the cases $\deg(v) = 1$ and $\deg(v) \ge 2$. When $\deg(v) = 1$, let $w$ be the neighbor of $v$. Then
\[ \homclass{v}{i}{T}{H} = \sum_{j=1}^k m_{i,j}\homclass{w}{j}{T-v}{H}.\]
Indeed, once we have colored $v$ with the representative vertex from $H^i$, there are, for each $j$, $m_{i,j}$ choices for a color from $H^j$ for $w$ (as the representative vertex in $H^i$ has $m_{i,j}$ neighbors in $H^j$ for each $j$). For each of these $m_{i,j}$ choices, recalling that $\homclass{w}{j}{T-v}{H}$ is independent of the choice of representative from $H^j$ due to \cref{lem:using_iso}, there are $\homclass{w}{j}{T-v}{H}$ ways to extend the coloring to the rest of $T$. More succinctly, we have
\begin{equation} \label{eq:treewalkalgorithm-pathcase} 
\mathbf{h}(T,v) = M \mathbf{h}(T-v,w)
\end{equation}
where $M$ is the automorphic similarity matrix of $H$. 

When $\deg(v) = d \ge 2$, let $w$ be any neighbor of $v$. Form tree $T'$ by taking the component of $T-v$ that contains $w$ and adding a new vertex $v' \sim w$. Form tree $T''$ by taking each component of $T-v$ that does not contain $w$ (of which there is at least one, because $v$ has degree at least two) and adding a new vertex $v''$ adjacent to each vertex that was adjacent to $v$ in $T$. Note that by gluing $v'$ to $v''$ we recover $T$. (See \cref{fig:decomposition}.)

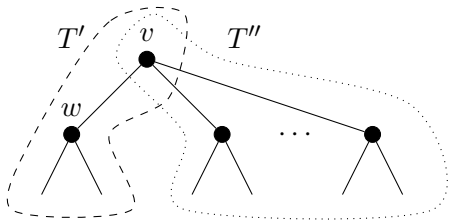
\begin{figure}[ht]
    \centering
    \begin{tikzpicture}
        \begin{scope}
            \node[circle,draw,fill,label=above:$v$, minimum size=6pt, inner sep=0pt] (v) at (1,1) {};
            \node[circle,draw,fill,label=$w$,minimum size=6pt, inner sep=0pt] (w) at (0,0) {};
            \node[circle,draw,fill,minimum size=6pt, inner sep=0pt] (w1) at (2,0) {};
            \node at (3,0) {$\cdots$};
            \node[circle,draw,fill,minimum size=6pt, inner sep=0pt] (w2) at (4,0) {};
            \node at (0,1.3) {$T'$};
            \node at (2.3,1.3) {$T''$};
    
            \draw (-0.4,-0.8) -- (w) -- (v) -- (w1) -- (2.4,-0.8);
            \draw (0.4,-0.8) -- (w);
            \draw (1.6, -0.8) -- (w1);
            \draw (v) -- (w2);
            \draw (3.6, -0.8) -- (w2) -- (4.4,-0.8);
    
            \draw[dashed] plot [smooth cycle] coordinates {(1.5,1.6) (0.5,1.5) (-0.5,0) (-0.8,-1) (0.8,-1) (0.5,0) (1.3,0.5)};
    
            \draw[dotted] plot [smooth cycle] coordinates {(0.6,1) (1.5,0) (1.5,-1) (4.8,-1) (4.5,0.4) (2,1) (1.1,1.6)};
        \end{scope}
    \end{tikzpicture}
    \caption{Deconstructing $T$ into $T'$ and $T''$.} \label{fig:decomposition}
\end{figure}

Then
\[ \mathbf{h}(T,v) = \mathbf{h}(T',v') \odot \mathbf{h}(T'',v''),\]
where $\odot$ represents the Hadamard product in which entries are multiplied term-wise. Indeed, to $H$-color $T$ while sending $v$ to $H^i$, we take any $H$-coloring of $T'$ where $v'$ is sent to $H^i$  ($\homclass{v'}{i}{T'}{H}$ many options) and independently any $H$-coloring of $T''$ where $v''$ is sent to $H^i$ ($\homclass{v''}{i}{T''}{H}$ many options). 
\end{proof}

We require an additional ingredient before proving \cref{thm:ordering_framework}. The KC ordering, introduced by Csikv\'ari~\cite{Csikvari2010} (under the name \emph{generalized tree shift}) as a generalization of an operation introduced by Kelmans~\cite{Kelmans1981}, is a partial ordering of trees on $n$ vertices. Let $T$ be a tree, and let $v_\ell \ne v_r$ be two non-leaf vertices with the property that the unique $v_\ell$ to $v_r$ path in $T$ is a bare path, meaning that each of the internal vertices, if they exist, have degree two. Let $P$ denote that path on $t \ge 2$ vertices.

Removing the edges and internal vertices (if any) of $P$ from $T$ leaves two components; call them $L$ (the one with $v_\ell$) and $R$ (the one with $v_r$). We think of these as the ``left'' and ``right'' components. Denote by $T^{KC}$ the tree built as follows: glue $L$ and $R$ together into a single tree by identifying the vertices $v_\ell$ and $v_r$; let $v$ be the vertex created by the identification. Then complete the construction of $T^{KC}$ by appending a path with $t-1$ new vertices to $v$; denote by $w$ the other end of the appended path. We say that $T^{KC}$ is obtained from $T$ by a {\it KC move}. \cref{fig:KC} illustrates this process. Note that $T$ has at least one such a pair of non-leaf vertices $v_\ell$, $v_r$ exactly if $T$ is not a star.

\begin{figure}[ht]
    \centering
    \begin{tikzpicture}
        \coordinate (L) at (-2.5,0);
        \draw[black, fill=white] (L) circle (1);
        \node at (L) {$L$};
        \coordinate (R) at (2.5,0);
        \draw[black, fill=white] (R) circle (1);
        \node at (R) {$R$};
        
        \coordinate (vl) at (-1.5,0);
        \node[label=above:{$v_\ell$}] at (-1.3,-0.1) {};
        \coordinate (vr) at (1.5,0);
        \node[label=above:{$v_r$}] at (1.3,-0.1) {};
        \coordinate (v1) at (-0.75,0);
        \coordinate (v2) at (0.75,0);
        \draw[fill=black] (vl) circle (3pt);
        \draw[fill=black] (vr) circle (3pt);
        \draw[fill=black] (v1) circle (3pt);
        \draw[fill=black] (v2) circle (3pt);
        \draw[fill=black, thick] (vl)--(v1);
        \draw[fill=black, thick, dashed] (v1) -- (v2);
        \draw[fill=black, thick] (v2) -- (vr);
        \draw [decorate,decoration={brace,amplitude=5pt,mirror,raise=10pt}] (-1.4,0) -- (1.4,0) node[midway,yshift=-23pt]{$P$};
        \node at (0,2) {};
        \node at (0,-2) {};
    \end{tikzpicture}
    \begin{tikzpicture}
        \node at (0,2) {};
        \node at (0,-2) {};
        \draw [line width=1pt, double distance=3pt,
             arrows = {-Latex[length=0pt 3 0]}] (0,0) -- (1.5,0);
    \end{tikzpicture}
    \begin{tikzpicture}
        \coordinate (L) at (-1,0);
        \draw[black, fill=white] (L) circle (1);
        \node at (L) {$L$};
        \coordinate (R) at (1,0);
        \draw[black, fill=white] (R) circle (1);
        \node at (R) {$R$};
        
        \coordinate (v) at (0,0);
        \node[label=right:{$v$}] at (v) {};
        \coordinate (w) at (0,-3);
        \node[label=below:{$w$}] at (w) {};
        \coordinate (v1) at (0,-0.75);
        \coordinate (v2) at (0,-2.25);
        \draw[fill=black] (v) circle (3pt);
        \draw[fill=black] (w) circle (3pt);
        \draw[fill=black] (v1) circle (3pt);
        \draw[fill=black] (v2) circle (3pt);
        \draw[fill=black, thick] (v)--(v1);
        \draw[fill=black, thick, dashed] (v1) -- (v2);
        \draw[fill=black, thick] (v2) -- (w);
        \draw [decorate,decoration={brace,amplitude=5pt,mirror,raise=10pt}] (0,-0.1) -- (0,-2.9) node[midway,xshift=-23pt]{$P$};
    \end{tikzpicture}
    \caption{A demonstration of a KC move.} \label{fig:KC}
\end{figure}
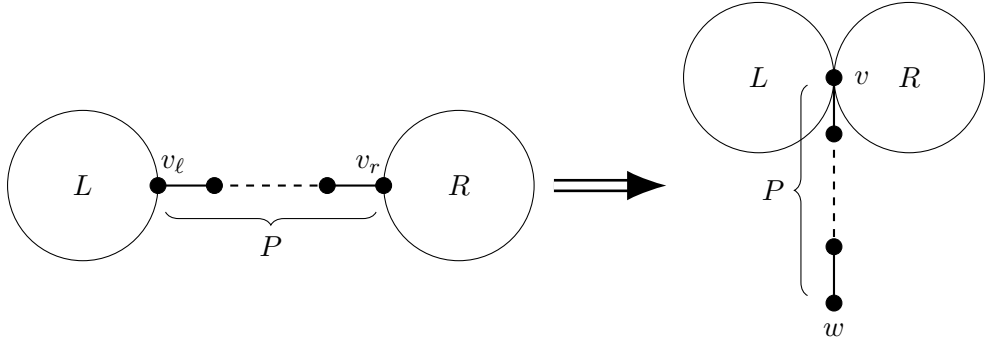

Csikv\'ari~\cite[Remark 2.4]{Csikvari2010} shows that the relation $S \le T$ defined by $T$ being obtainable from $S$ by a sequence of KC moves defines, for each $n$, a partial order on ${\mathcal T}_n$ (the {\it KC ordering}), with $P_n$ the unique minimal element and $S_n$ the unique maximal element.

The following lemma is necessary for the proof of \cref{thm:ordering_framework}. The proof that we give follows \cite[Theorem 1.1, $(a) \Rightarrow (c)$]{KeilsonKester1977}. We choose to present the full details because \cite{KeilsonKester1977} only treats matrices whose row sums are $1$. 

\begin{lemma} \label{lem:matrix_prop}
Let $M$ be a $k \times k$ non-negative matrix that has the increasing columns property, and let ${\mathbf h}$ be a non-negative column vector of dimension $k$ whose entries are non-decreasing. Then the entries of $M{\mathbf h}$ are non-negative and non-decreasing.
\end{lemma}

\begin{proof}
Let $S$ be the $k \times k$ matrix that has $1$s on and below the main diagonal and $0$s everywhere else. Note that $S^{-1}$ has $1$s down the main diagonal, $-1$s down the first subdiagonal, and $0$s everywhere else, which can be verified by computing $SS^{-1}$.

Observe that for a vector ${\mathbf v}=\begin{bmatrix} v_1 & v_2 & \cdots & v_k \end{bmatrix}^T$, the entries of ${\mathbf v}$ being non-negative and non-decreasing is equivalent to the entries of $S^{-1}{\mathbf v}$ being non-negative because the first entry of $S^{-1}{\mathbf v}$ is $v_1$ and for $i > 1$, the \ith{i} entry is $v_i-v_{i-1}$. 

We wish to show that the entries of $M{\mathbf h}$ are non-negative and non-decreasing. This is implied by $S^{-1}M{\mathbf h}$ having non-negative entries, which is in turn implied by $(S^{-1}MS)(S^{-1}{\mathbf h})$ having non-negative entries. By assumption, the entries of ${\mathbf h}$ are non-negative and non-decreasing, so the entries of $S^{-1}{\mathbf h}$ are non-negative. Therefore it suffices to show that $S^{-1}MS$ has non-negative entries. 

Let $A=S^{-1}MS$. We have
\begin{eqnarray*}
A_{i,j} & = & \sum_{p=1}^k (S^{-1})_{i,p} \sum_{q=1}^k m_{p,q}S_{q,j} \\
& = & \sum_{p=1}^k (S^{-1})_{i,p} \sum_{q=1}^j m_{p,q} \\
& = & \sum_{q=1}^j m_{i,q} - \sum_{q=1}^j m_{i-1,q},
\end{eqnarray*}
where we adopt the convention that $m_{0,q}=0$. Then
\[\sum_{q=1}^j m_{i,q} - \sum_{q=1}^j m_{i-1,q} \ge 0\]
follows from the increasing columns property. 
\end{proof}

We're now prepared to prove \cref{thm:ordering_framework}, which we restate here for convenience.

\orderings*

\begin{proof}
We show that when $H$ is a graph whose automorphic similarity classes can be ordered in such a way that the automorphic similarity matrix of $H$ has the increasing columns property, applying a non-trivial KC move to a tree does not decrease the number of $H$-colorings of the tree. As the path is the minimum element of the KC ordering, it must therefore be among the minimizers of $\hom(T_n,H)$.

We begin by establishing some notation. Let $T$ be a non-star tree, let $v_\ell \ne v_r$ be two non-leaf vertices of $T$ with the property that all of the internal vertices on the unique $v_\ell$ to $v_r$ path have degree two, let there be $t$ vertices in that internal path, and let $T^{KC}$ be the tree obtained from $T$ by applying a KC move at vertices $v_\ell$ and $v_r$. Suppose $H$ has automorphic similarity classes $H^1, \ldots, H^k$, labeled in such a way that the automorphic similarity matrix $M$ has the increasing columns property. For $1 \le i \le k$, denote by ${\mathcal L}_i$ the set of $H$-colorings of $L$ in which $v_\ell$ is mapped to some specific (but arbitrarily chosen) vertex $w_i$ of $H^i$, and let $\ell_i=|{\mathcal L}_i|$. Define ${\mathcal R}_i$ and $r_i$ analogously (with $L$ replaced by $R$). Note that by \cref{lem:using_iso}, $\ell_i$ and $r_i$ depend only on $i$ and not on the specific choice of $w_i$. For $w, v \in V(H)$, denote by ${\mathcal P}^t(w,v)$ the set of $H$-colorings of any path $x_1, \ldots, x_t$ (where, recall, $t$ is the number of vertices on the unique path in $T$ between $v_\ell$ and $v_r$) in which $x_1$ is mapped to $w$ and $x_t$ is mapped to $v$, and let $p^t(w,v)=|{\mathcal P}^t(w,v)|$. Finally, for $1 \le i \le k$ and $1 \le j \le k$, let ${\mathcal P}^t_{i,j} = \bigcup_{w \in H^i,~ v \in H^j} {\mathcal P}^t(w,v)$ and let $p^t_{i,j}=|{\mathcal P}^t_{i,j}| =\sum_{w \in H^i,~v \in H^j} p^t(w,v)$. 

The idea at the core of this proof is that $\hom(T,H)$ and $\hom(T^{KC},H)$ can both be expressed in terms of the $\ell_i$, $r_j$, and $p^t_{i,j}$s by partitioning the $H$-colorings according to, in the case of $T$, the automorphic similarity classes to which $v_\ell$ and $v_r$ are mapped, and in the case of $T^{KC}$, the automorphic similarity classes to which $v$ and $w$ are mapped. In each case we get expressions that are sums of $k^2$ terms, and we show $\hom(T^{KC},H) \ge \hom(T,H)$ by proving each term in the summation expression for $\hom(T^{KC},H) - \hom(T,H)$ is non-negative.

When computing $\hom(T,H)$, $v_\ell$ and $v_r$ can take any pair of values, and those pairs of values then determine the values taken by the endvertices of $P$, which gives
\begin{eqnarray} 
\hom(T,H) & = & \sum_{i=1}^k \sum_{j=1}^k \sum_{w \in H^i} \sum_{v \in H^j} \ell_ir_j p^t(v,w) \notag \\
& = & \sum_{i=1}^k \sum_{j=1}^k \ell_ir_j p^t_{i,j}. \label{eq:pre_move_hom}
\end{eqnarray}
Note here that the count of $H$-colorings of $T$ that send $v_\ell$ to some $w \in H^i$ and $v_r$ to some $v \in H^j$ should, {\it a priori}, involve terms that depend on $w$ and $v$, but by \cref{lem:using_iso}, these terms ($\ell_i$ and $r_j$, respectively) are independent of $w$ and $v$. 

In contrast, when computing $\hom(T^{KC},H)$, $v_\ell$ and $v_r$ are forced to take a common value, which determines the values at one of the endvertices of $P$, and the value at the other endvertex is free, which means
\begin{eqnarray} 
\hom(T^{KC},H) & = & \sum_{i=1}^k \sum_{j=1}^k \sum_{w \in H^i} \sum_{v \in H^j} \ell_ir_i p^t(w,v) \notag \\
& = & \sum_{i=1}^k \sum_{j=1}^k \ell_ir_i p^t_{i,j}. \label{eq:post_move_hom}
\end{eqnarray}
 
The term $\ell_ir_ip^t_{i,i}$ appears in both sums for each $1 \le i \le k$. Using symmetry to note $p^t_{i,j} = p^t_{j,i}$ for $1 \le i < j \le k$, we combine equations \eqref{eq:pre_move_hom} and \eqref{eq:post_move_hom} to write
\begin{eqnarray}
\hom(T^{KC},H) - \hom(T,H) & = & \sum_{1 \le i < j \le k} (\ell_ir_i+\ell_jr_j-\ell_ir_j-\ell_jr_i)p^t_{i,j} \notag \\
& = & \sum_{1 \le i < j \le k} (\ell_j-\ell_i)(r_j-r_i)p^t_{i,j}. \label{eq:non-negative}
\end{eqnarray}
Noting that trivially $p^t_{i,j}\ge 0$, we conclude that in order to prove $\hom(T^{KC},H) \ge \hom(T,H)$ it suffices to show that for each $i,j$ we have $(\ell_j-\ell_i)(r_j-r_i) \ge 0$.

Given any tree $T$ and distinguished vertex $v$, recall that $\homclass{v}{i}{T}{H}$ denotes the number of $H$-colorings of $T$ in which distinguished vertex $v$ is mapped to some specific (but arbitrarily chosen) vertex $w_i$ of $H^i$, where we again note that by \cref{lem:using_iso}, $\homclass{v}{i}{T}{H}$ depends only on $i$ and not on the specific choice of $w_i$. As in the statement of \cref{lem:tree_walk}, define $\mathbf{h}(T,v)$ to be the $k$-dimensional column vector whose \ith{i} entry is $\homclass{v}{i}{T}{H}$. We now prove that our condition on the columns of $M$, the automorphic similarity matrix of $H$, ensures that $\mathbf{h}(T,v)$ is non-decreasing for any choice of $T$ and $v$, so that in particular when $i < j$,
\[ \ell_j = \homclass{v_\ell}{j}{T}{H} \ge \homclass{v_\ell}{i}{T}{H} = \ell_i,\]
and analogously $r_j \ge r_i$.

To see that $\mathbf{h}(T,v)$ is always non-decreasing, we use \cref{lem:tree_walk,lem:matrix_prop}. By the proof of \cref{lem:tree_walk}, $\mathbf{h}(T,v)$ can be computed from the all-$1$s column vector by a sequence of steps that involve taking Hadamard products of vectors and pre-multiplying vectors by $M$. Combining the facts that the all-$1$s vector is non-negative and non-decreasing, that the Hadamard product of non-negative, non-decreasing vectors is both non-negative and non-decreasing, and that (by \cref{lem:matrix_prop}) pre-multiplying by $M$ preserves the properties of being non-negative and non-decreasing, we see that $\mathbf{h}(T,v)$ is indeed non-negative and non-decreasing. Therefore, $\ell_j \ge \ell_i$ and $r_j \ge r_i$, so for each $i$ and $j$, $(\ell_j-\ell_i)(r_j-r_i) \ge 0$ and we must have $\hom(T^{KC},H) \ge \hom(T,H)$, as desired.
\end{proof}

We now restate and prove \cref{cor:uniquenessalt}, which gives conditions under which paths uniquely minimize $\hom(T,H)$. The proof focuses on equation \eqref{eq:non-negative} in the special case where $T$ is a path, and aims to show that at least one of the summands on the right-hand of equation \eqref{eq:non-negative} is (under certain circumstances) strictly positive. As an aside, we note that unlike traditional corollaries which follow directly from a theorem, we justify \cref{cor:uniquenessalt} using ideas from the proof of \cref{thm:ordering_framework} rather than using the statement. Nonetheless, we use the term ``corollary'' to highlight the connection between the results. 

\uniquenessalt*

\begin{proof}
Let $P^{KC}_n$ be obtained from $P_n$ by a KC move. Suppose that the bare path involved in this move has $t$ vertices. We have 
\begin{eqnarray}
\hom(P^{KC}_n,H) - \hom(P_n,H) & = & \sum_{i< j} (\ell_j-\ell_i)(r_j-r_i)p^t_{ij} \label{eq:line1}\\
& \ge & (\ell_{b(t)}-\ell_{a(t)})(r_{b(t)}-r_{a(t)})p^t_{a(t),b(t)}. \label{inq:line2}
\end{eqnarray}
Equation~\eqref{eq:line1} is an instance of equation \eqref{eq:non-negative}, while inequality~\eqref{inq:line2} uses that $(\ell_j-\ell_i)(r_j-r_i)p^t_{ij} \ge 0$ for all $i, j$, which comes from the proof of \cref{thm:ordering_framework}. 

Note that $p^t_{a(t),b(t)}$ counts the number of $H$-colorings of $P_t$ that send one end vertex of $P_t$ to something in $H^{a(t)}$ and the other end to something in $H^{b(t)}$; by hypothesis there is at least one such $H$-coloring, so $p^t_{a(t),b(t)} > 0$.

We have that $\ell_{b(t)} = \homclass{v}{b(t)}{L}{H}$ and $\ell_{a(t)} = \homclass{v}{a(t)}{L}{H}$. As $L$ is a path, we have by hypothesis that $\ell_{b(t)} > \ell_{a(t)}$. Similarly, $r_{b(t)} > r_{a(t)}$. We conclude that
\[(\ell_{b(t)}-\ell_{a(t)})(r_{b(t)}-r_{a(t)})p^t_{a(t),b(t)} > 0.\]

It follows from inequality \eqref{inq:line2} that $\hom(P^{KC}_n,H) - \hom(P_n,H) > 0$, and so for any $T_n$ obtained from $P_n$ by a sequence of KC moves, $\hom(P_n,H)<\hom(P_n^{KC},H)\le\hom(T_n,H)$. Since all trees on $n$ vertices can be thus obtained, the result follows. 
\end{proof}

\section{Applications of Theorem~\ref{thm:ordering_framework}} \label{sec:ordering_applications}

One goal of this paper is to use \cref{thm:ordering_framework} to contribute towards the resolution of~\cref{prob:CL}. Before we state our results in this direction, we summarize some pertinent results from the literature. We then show how \cref{thm:ordering_framework} can reprove and sometimes strengthen some of these results, before moving on to establishing some new Hoffman-London families.

\subsection{Two previous results on Hoffman-London graphs}

In inequality \eqref{inq:PT}, we saw that the path minimizes the number of independent sets among trees, i.e. that the target graph $H_{\rm ind}$ is Hoffman-London, where, recall, $H_{\rm ind}$ consists of two vertices joined by an edge with a loop at one vertex. Engbers and Galvin~\cite{EngbersGalvin2017} gave a significant generalization of this result by showing that target graphs $H$ formed by adding looped dominating vertices (looped vertices that are adjacent to all other vertices) to a regular graph are Hoffman-London, and they characterized the strongly Hoffman-London target graphs in this family.
\begin{theorem}[Engbers, Galvin~\cite{EngbersGalvin2017}] \label{thm:dom_reg}
If $H$ is obtained from a regular graph by adding any number of looped dominating vertices, then $H$ is Hoffman-London. If $H$ is not regular (equivalently, if $H$ is not a fully looped complete graph), then $H$ is strongly Hoffman-London. 
\end{theorem}

A special case of \cref{thm:dom_reg}---that $H$ is strongly Hoffman-London when it is obtained from an empty (edgeless) graph by adding any non-zero number of looped dominating vertices, or equivalently when it is the join of empty graph and a fully looped complete graph---can also be easily derived from earlier work of Wingard \cite[Theorem 5.1 and Theorem 5.2]{Wingard1995} on independent sets of fixed size in trees. 

Csikv\'ari and Lin~\cite{CsikvariLin2015} showed that paths and stars are Hoffman-London.
\begin{theorem}[Csikv\'ari, Lin~\cite{CsikvariLin2015}] \label{thm:paths_and_stars_as_H}
If $H$ is either a path or a star (on any number of vertices), then for all $n$ and all $T_n \in {\mathcal T}_n$ we have
\[\hom(P_n,H) \le \hom(T_n,H).\]
\end{theorem}
Note that Csikv\'ari and Lin do not comment on when paths and stars are strongly Hoffman-London. We will address this question for paths in \cref{ssec:paths}, specifically after the proof of \cref{thm:even_path_minimizer}. For stars, note that \cref{thm:complete_bipartite} shows that $S_k$ is not strongly Hoffman-London for any $k$.

\subsection{Target graphs with two automorphic similarity classes} \label{ssec:two_auto_classes}

In this section we specialize \cref{thm:ordering_framework} to the case in which $H$ has exactly two automorphic similarity classes and is not a regular graph. (Note that by \cref{obs:regular}, if $H$ is regular, then $\hom(T_n,H)$ is independent of $T_n$, so $H$ is Hoffman-London but not strongly Hoffman-London.) Recall that automorphic similarity classes partition the vertex set, and so in this case all vertices of $H$ are in either $H^1$ or $H^2$. We have the following corollary of \cref{thm:ordering_framework} and \cref{cor:uniquenessalt}. 
\begin{corollary} \label{cor:two-auto}
Let $H$ be a target graph that is not regular and that has two automorphic similarity classes, $H^1$ and $H^2$, ordered so that the (common) degree of any vertex in $H^1$ is smaller that the (common) degree of any vertex in $H^2$. If $m_{1,2} \le m_{2,2}$, then $H$ is Hoffman-London. If, further, there is a vertex in $H$ that has a neighbor in $H^1$ and a neighbor in $H^2$, then $H$ is strongly Hoffman-London.
\end{corollary}

\begin{proof}
To see that $H$ is Hoffman-London, note that for any vertices $x \in H^1$ and $y \in H^2$ we have  
\begin{equation} \label{inq:two_classes}
\deg(x) = m_{1,1}+m_{1,2} < m_{2,1}+m_{2,2} = \deg(y).
\end{equation}
Together with $m_{1,2} \le m_{2,2}$, inequality \eqref{inq:two_classes} shows that the automporphic similarity matrix of $H$ satisfies the increasing columns property and so is Hoffman-London by \cref{thm:ordering_framework}.

We now use \cref{cor:uniquenessalt} to establish that if $H$ has a vertex that is adjacent to vertices in both $H^1$ and $H^2$, then $H$ is strongly Hoffman-London. Observe that for $0 < a \le b$ we have
\begin{equation} \label{eq:two_classes}
\begin{bmatrix}
m_{1,1} & m_{1,2} \\
m_{2,1} & m_{2,2}
\end{bmatrix}
\begin{bmatrix}
a \\
b
\end{bmatrix}
=
\begin{bmatrix}
a(m_{1,1}+m_{1,2}) + (b-a) m_{1,2} \\
a(m_{2,1}+m_{2,2}) + (b-a) m_{2,2}
\end{bmatrix}
:= 
\begin{bmatrix}
a' \\
b'
\end{bmatrix},
\end{equation}
and our hypotheses that $m_{1,1}+m_{1,2} < m_{2,1}+m_{2,2}$ and $0 < a \le b$ imply that $0 < a' < b'$. We use this fact to show that for all $s \ge 2$, we have $\homclass{v}{1}{P_s}{H} < \homclass{v}{2}{P_s}{H}$. As in the statement of \cref{lem:tree_walk}, define $\mathbf{h}(T,v)$ to be the $k$-dimensional column vector whose \ith{i} entry is $\homclass{v}{i}{T}{H}$. When $s = 1$, the \ith{i} entry of $\mathbf{h}(P_1,v)$ is $\homclass{v}{i}{P_1}{H}$, the count of $H$-colorings of $P_1$ in which $v$ is mapped to a representative of $H^1$, so $\mathbf{h}(P_1,v) = \begin{bmatrix}1 & 1\end{bmatrix}^\trans$, where $v$ is the unique vertex of $P_1$. Then for $s \ge 2$, we have $\mathbf{h}(P_s,v) = M \mathbf{h}(P_{s-1},v)$, so by equation~\eqref{eq:two_classes}, we conclude $\homclass{v}{1}{P_s}{H} < \homclass{v}{2}{P_s}{H}$.

To apply \cref{cor:uniquenessalt} and conclude that $H$ is strongly Hoffman-London, it remains to show that for all $t \ge 2$ there is an $H$-coloring of $P_t$ that sends one end vertex to $H_1$ and the other to $H_2$. Let $x \in H$ be the vertex in $H$ that has a neighbor $y \in H^1$ and $z \in H^2$, noting that $x, y$, and $z$ need not necessarily be distinct. Suppose $x \in H^1$. If $t$ is even, then we can map the vertices of $P_t$ alternately to $x$ and $z$; if $t$ is odd, then we can map the first vertex of $P_t$ to $y$, the second to $x$, and then proceed as in the even case. The argument for $x \in H^2$ follows analogously.    
\end{proof}

We can use \cref{cor:two-auto} to recover a portion of \cref{thm:dom_reg}. Specifically, \cref{cor:two-auto} yields \cref{thm:dom_reg} in the case when looped dominating vertices are added to a vertex-transitive (and so necessarily regular) graph. Although the result below is weaker than \cref{thm:dom_reg}, we include it as it makes our classification of target graphs on three vertices (see \cref{sec:constraint-3-vertices}) self-contained. 

\begin{proof}[Proof of special case of \cref{thm:dom_reg}]
Let $H$ be obtained from a vertex transitive graph, say $H_0$, by adding some number of looped dominating vertices. If either the number of looped dominating vertices added is zero or $H_0$ is a fully looped complete graph, then $H$ is regular, and so by \cref{obs:regular} $H$ is Hoffman-London but not strongly Hoffman-London. For the remainder of the proof, we assume that $H_0$ is regular but not a fully looped complete graph and at least one looped dominating vertex is added. Let $A$ be the vertex set of $H_0$ and let $B$ be the set of added looped dominating vertices. Let $k$ be the degree of vertices in $H_0$; note $k < |A|$.     

It is easy to check that $H$ has two automorphic similarity classes, namely $H^1=A$ and $H^2=B$. 
Let $x$ and $y$ be arbitrary vertices in $H^1$ and $H^2$, respectively. We have 
\[
\deg(x) = k + |B| < |A| + |B| = \deg(y)
\]
and 
\[
m_{1,2} = |B| = m_{2,2}. 
\]
We apply \cref{cor:two-auto} to conclude that $H$ is Hoffman-London. Furthermore, because any vertex in $H^2$ is adjacent to itself and to all vertices in $H^1$, \cref{cor:two-auto} also shows that $H$ is strongly Hoffman-London.
\end{proof}

We note that \cref{thm:ordering_framework} alone could not possibly imply the full strength of \cref{thm:dom_reg}, which we demonstrate in \cref{ex:folkman}.

\begin{example}\label{ex:folkman}

Consider the Folkman graph~\cite{Folkman1967}, a $4$-regular graph on $20$ vertices constructed from a $K_5$ by first subdividing each edge and then cloning each of the vertices of the original $K_5$; see \cref{fig:folkman} where the vertices of the original $K_5$ are white and those from subdivided edges are black.

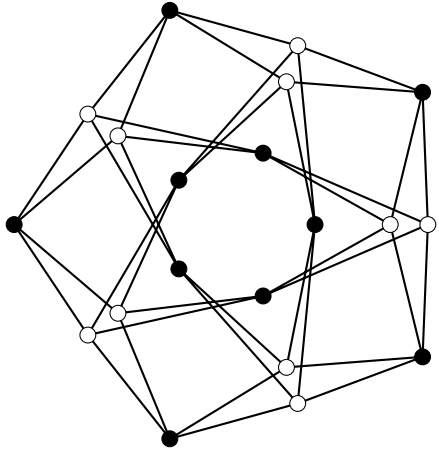
\begin{figure}[ht]
    \centering
    \begin{tikzpicture}
        \coordinate (w1) at (1,0);
        \coordinate (w2) at (.31,.95);
        \coordinate (w3) at (-.81,.59);
        \coordinate (w4) at (-.81,-.59);
        \coordinate (w5) at (.31,-.95);

        \coordinate (v11) at (2,0);
        \coordinate (v12) at (2.5,0);
        \coordinate (v21) at (0.62,1.9);
        \coordinate (v22) at (0.77,2.38);
        \coordinate (v31) at (-1.62,1.18);
        \coordinate (v32) at (-2.02,1.47);
        \coordinate (v41) at (-1.62,-1.18);
        \coordinate (v42) at (-2.02,-1.47);
        \coordinate (v51) at (0.62,-1.9);
        \coordinate (v52) at (0.77,-2.38);

        \coordinate(x1) at (2.43,1.76);
        \coordinate(x2) at (-.93,2.85);
        \coordinate(x3) at (-3,0);
        \coordinate(x4) at (-.93,-2.85);
        \coordinate(x5) at (2.43,-1.76);

        \draw[thick,black] (w1) -- (v21) -- (w3) -- (v41) -- (w5) -- (v11) -- (w2) -- (v31) -- (w4) -- (v51) -- (w1) -- (v22) -- (w3) -- (v42) -- (w5) -- (v12) -- (w2) -- (v32) -- (w4) -- (v52) -- (w1);
        \draw[thick,black] (x1) -- (v21) -- (x2) -- (v31) -- (x3) -- (v41) -- (x4) -- (v51) -- (x5) -- (v11) -- (x1) -- (v22) -- (x2) -- (v32) -- (x3) -- (v42) -- (x4) -- (v52) -- (x5) -- (v12) -- (x1);

        \draw[fill=black] (w1) circle (3pt);
        \draw[fill=black] (w2) circle (3pt);
        \draw[fill=black] (w3) circle (3pt);
        \draw[fill=black] (w4) circle (3pt);
        \draw[fill=black] (w5) circle (3pt);

        \draw[fill=white,draw=black] (v11) circle (3pt);
        \draw[fill=white,draw=black] (v12) circle (3pt);
        \draw[fill=white,draw=black] (v21) circle (3pt);
        \draw[fill=white,draw=black] (v22) circle (3pt);
        \draw[fill=white,draw=black] (v31) circle (3pt);
        \draw[fill=white,draw=black] (v32) circle (3pt);
        \draw[fill=white,draw=black] (v41) circle (3pt);
        \draw[fill=white,draw=black] (v42) circle (3pt);
        \draw[fill=white,draw=black] (v51) circle (3pt);
        \draw[fill=white,draw=black] (v52) circle (3pt);

        \draw[fill=black] (x1) circle (3pt);
        \draw[fill=black] (x2) circle (3pt);
        \draw[fill=black] (x3) circle (3pt);
        \draw[fill=black] (x4) circle (3pt);
        \draw[fill=black] (x5) circle (3pt);
    \end{tikzpicture}
    \caption{The Folkman graph.} \label{fig:folkman}
\end{figure} 

The Folkman graph is bipartite and has two automorphic similarity classes, specifically the two bipartition classes. By \cref{thm:dom_reg}, adding a looped dominating vertex to the Folkman graph produces an $H$ that is Hoffman-London. This $H$ has three automorphic similarity classes: the two bipartition classes of the Folkman graph, each of size ten, which we call $A$ and $B$, and the looped dominating vertex, which we call $C$. Vertices in $A$ (respectively, $B$) have no neighbors in $A$ (respectively, $B$), four neighbors in $B$ (respectively, $A$), and one neighbor in $C$. The vertex in $C$ has ten neighbors in $A$, ten in $B$ and one in $C$. So if we set $A=H^1$, $B=H^2$ (or $A=H^2$, $B=H^1$) and $C=H^3$, the corresponding automorphic similarity matrix is 
\[
M = \begin{bmatrix}
0 & 4 & 1 \\
4 & 0 & 1 \\
10 & 10 & 1
\end{bmatrix},
\]
which fails the increasing columns property of \cref{thm:ordering_framework}. Setting $C=H^1$ or $C=H^2$ also produces automorphic similarity matrices that fail the increasing columns property, meaning that \cref{thm:ordering_framework} cannot be used to conclude that $H$ is Hoffman-London.
\end{example}

As an example of the utility of \cref{cor:two-auto} to obtain new families of Hoffman-London graphs, consider the family of constraint graphs $H(a,b,\ell)$, defined for $a, b \ge 1$ and $\ell \ge 0$. Start with a clique on $b$ vertices. At each vertex $v$ of the clique, append $\ell$ copies of a clique on $a$ vertices, with $v$ the only vertex in common to the $\ell$ appended cliques; see \cref{fig:habl}. This family generalizes many well known families of graphs; for example, $H(2,1,\ell)$ is the star graph, $H(2,2,\ell)$ is the {\it balanced double star}---an edge with $\ell$ edges appended to each endpoint, and $H(3,1,\ell)$ is the {\it fan graph}---a collection of triangles sharing a single common vertex. More generally, $H(a,1,\ell)$ is a collection of $\ell$ cliques on $a$ vertices each, all sharing a single common vertex; following standard notation from topology, we call this a {\it bouquet of $\ell$ $a$-cliques}. 

\begin{figure}[ht]
    \centering
    \begin{tikzpicture}
        \coordinate (Kb) at (0,0);
        \draw[black, fill=white] (Kb) circle (1.5);
        \node at (Kb) {$K_b$};

        \draw[thick,dotted] (1.4,2.5) to [out=-5,in=100] (2,1.9);
        \coordinate (v1) at (0.75,1.3);
        \draw[fill=black] (v1) circle (3pt);
        \coordinate (v2) at (1.39,-0.57);
        \draw[fill=black] (v2) circle (3pt);
        \draw[thick,black] (v1) to [out=105,in=150,loop,style={min distance=100pt}] (v1);
        \draw[thick,black] (v1) to [out=60,in=105,loop,style={min distance=100pt}] (v1);
        \draw[thick,black] (v1) to [out=-15,in=30,loop,style={min distance=100pt}] (v1);
        \node at (0,2.2) {$K_a$};
        \node at (0.9,2.5) {$K_a$};
        \node at (2,1.5) {$K_a$};
        \node[label=right:{$\ell$ times}] at (1.7,2.4) {};

        \draw[thick,dotted] (2.6,-1.3) to [out=-95,in=0] (2,-1.7);
        \draw[thick,black] (v2) to [out=10,in=55,loop,style={min distance=100pt}] (v2);
        \draw[thick,black] (v2) to [out=-35,in=10,loop,style={min distance=100pt}] (v2);
        \draw[thick,black] (v2) to [out=-100,in=-55,loop,style={min distance=100pt}] (v2);
        \node at (2.3,0) {$K_a$};
        \node at (2.7,-0.9) {$K_a$};
        \node at (1.7,-1.9) {$K_a$};
        \node[label=right:{$\ell$ times}] at (2.3,-1.7) {};
    \end{tikzpicture}
    \caption{A visualization of the family of graphs $H(a,b,\ell)$.} \label{fig:habl}
\end{figure}
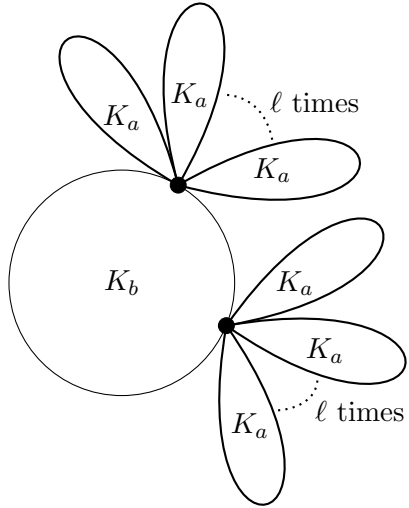

For some choices of parameters, it is straightforward to verify whether $H(a,b,\ell)$ is Hoffman-London. In the cases $\ell=0$, $\ell \ge 1$ and $a=1$, or $\ell=1$ and $b=1$, we see $H(a,b,\ell)$ is just a complete graph, which is regular and so Hoffman-London but not strongly Hoffman-London. If $\ell \ge 2$, $a = 2$, and $b=1$ then $H(a,b,\ell)$ is, as observed above, a star, and so is Hoffman-London by \cref{thm:paths_and_stars_as_H}, although, as we saw in \cref{thm:complete_bipartite}, it is not strongly Hoffman-London. For many other choices of parameters, we can use \cref{cor:two-auto}.      

\begin{corollary} \label{cor:clique-creatures}
For all $a, b \ge 2$ and $\ell \ge 1$, the target graph $H(a,b,\ell)$ is strongly Hoffman-London.
\end{corollary}

\begin{proof}
The graph $H(a,b,\ell)$ has two automorphic similarity classes: $H^2$ consisting of the $b$ vertices of the initial clique and $H^1$ consisting of all $\ell(a-1)$ added vertices. We have $m_{1,2} = 1 \le b-1 = m_{2,2}$, and if $x$ is any vertex in $H^1$ and $y$ any vertex in $H^2$ we have $\deg(x) = a-1 < b-1 + \ell(a-1) = \deg(y)$. We conclude that $H$ is Hoffman-London by \cref{cor:two-auto}. 

Any vertex in $H^2$ is adjacent to at least one other vertex in $H^2$ and at least one vertex in $H^1$, and so \cref{cor:two-auto} also demonstrates that $H$ is strongly Hoffman-London.
\end{proof}

The graphs $H(a,b,\ell)$ that are not covered by \cref{cor:clique-creatures} or the preceding discussion are those with $a \ge 3$, $b=1$, and $\ell \ge 2$. In these cases, as mentioned above, $H$ is a bouquet of $\ell$ $a$-cliques. Such target graphs have two automorphic similarity classes: one that includes only the central vertex and another that includes all other vertices. For this collection of target graphs, we cannot apply \cref{thm:ordering_framework}: if we choose the central vertex to be $H^1$, then the associated automorphic similarity matrix has $\ell(a-1) > 1$ as its $(1,2)$-entry and $1$ as its $(2,2)$-entry and fails the increasing columns property, while if we choose the central vertex to be $H^2$ then the associated automorphic similarity matrix has $1$ as its $(1,2)$-entry and $0$ as its $(2,2)$-entry, and so again fails the increasing columns property. It is unknown which graphs of this subfamily, if any, are Hoffman-London.

\subsection{Paths and looped paths as target graphs} \label{ssec:paths}

In this section, we consider paths and fully looped paths as target graphs. We begin by using \cref{thm:ordering_framework} to recover part of \cref{thm:paths_and_stars_as_H}, specifically the case when $H$ is a path with an even number of vertices. We also strengthen this portion of \cref{thm:paths_and_stars_as_H} by using \cref{cor:uniquenessalt} to show that target graphs $H$ in this case with at least four vertices are strongly Hoffman-London.

\begin{theorem} \label{thm:even_path_minimizer}
The path graph $P_{2m}$ is Hoffman-London for all $m \ge 1$ and strongly Hoffman-London for all $m \ge 2$.
\end{theorem}

\begin{proof}
Let $H = P_{2m}$ for some $m \ge 1$, specifically with vertices $u_1, \ldots, u_{2m}$ and edges $u_iu_{i+1}$ for $1 \le i \le 2m-1$. We claim that $P_{2m}$ admits exactly two automorphisms: the identity and the {\it reversal} map that sends $u_i$ to $u_{2m+1-i}$ for $i=1, \ldots, 2m$. To verify this claim, first note that an automorphism of $P_{2m}$ must preserve degrees and thus leaves must be mapped to leaves. If $u_1$ is mapped to $u_1$, then $u_2$, the unique neighbor of $u_1$, must be mapped to $u_2$, $u_3$ must be mapped to $u_3$, and so on, resulting in the identity map. If $u_1$ is instead mapped to $u_{2m}$, then $u_2$ must be mapped to $u_{2m-1}$, $u_3$ must be mapped to $u_{2m-2}$, and so on, giving resulting in the reversal map. It follows that $P_{2m}$ has $m$ automorphic similarity classes, each consisting of a pair of vertices $\{u_i,u_j\}$ such that $i+j=2m+1$. 

Set $H^i=\{u_i, u_{2m+1-i}\}$ for $i=1, \ldots, m$. The automorphic similarity matrix of $H$ with respect to this ordering of the automorphic similarity classes has the form 
\[ \begin{bmatrix}
0 & 1 & 0 & 0 & \cdots & 0 & 0 & 0\\
1 & 0 & 1 & 0 & \cdots & 0 & 0 & 0\\
0 & 1 & 0 & 1 & \cdots & 0 & 0 & 0\\
\vdots & \vdots & \vdots & \vdots & \ddots & \vdots & \vdots& \vdots\\
0 & 0 & 0 & 0 & \cdots & 1 & 0 & 1\\
0 & 0 & 0 & 0 & \cdots & 0 & 1 & 1
    \end{bmatrix}. \]
That is, it has $1$s on the first subdiagonal, $1$s on the first superdiagonal, a $1$ in the $(m,m)$ position, and $0$s everywhere else. It is straightforward to check that all terminal sums of the \ith{i} row are less than or equal to the corresponding terminal sums of the \ith{(i+1)} row for all $i < m$, and so this matrix satisfies the increasing columns property, establishing that $P_{2m}$ is Hoffman-London by \cref{thm:ordering_framework}.

We now turn to establishing that $P_{2m}$ is strongly Hoffman-London for $m \ge 2$. We seek to apply \cref{cor:uniquenessalt}. Towards this goal, we will show by induction on $t$ that
\begin{equation} \label{inq:paths}
\homclass{v}{1}{P_t}{P_{2m}} < \homclass{v}{2}{P_t}{P_{2m}} \le \homclass{v}{3}{P_t}{P_{2m}} \le \cdots \le \homclass{v}{m}{P_t}{P_{2m}},
\end{equation}
where $v$ is an endvertex of $P_t$.

When $t=2$, inequality \eqref{inq:paths} holds because $\homclass{v}{1}{P_2}{P_{2m}}=1$ and $\homclass{v}{i}{P_2}{P_{2m}}=2$ for $2 \le i \le m$. For $t > 2$, let $v'$ be the unique neighbor of $v$ in $P_t$. Using equation \eqref{eq:treewalkalgorithm-pathcase} from the proof of \cref{lem:tree_walk}, we have the recurrence relations
\begin{eqnarray*}
\homclass{v}{1}{P_t}{P_{2m}} & = & \homclass{v'}{2}{P_{t-1}}{P_{2m}}, \\
\homclass{v}{i}{P_t}{P_{2m}} & = & \homclass{v'}{i-1}{P_{t-1}}{P_{2m}} + \homclass{v'}{i+1}{P_{t-1}}{P_{2m}} \text{ when } 2 \le i \le m-1, \text{ and} \\
\homclass{v}{m}{P_t}{P_{2m}} & = & \homclass{v'}{m-1}{P_{t-1}}{P_{2m}} + \homclass{v'}{m}{P_{t-1}}{P_{2m}}.
\end{eqnarray*}
By induction, we have
\begin{equation}\label{inq:path_induction}
\homclass{v'}{1}{P_{t-1}}{P_{2m}} < \homclass{v'}{2}{P_{t-1}}{P_{2m}} \le \cdots \le \homclass{v'}{m}{P_{t-1}}{P_{2m}}.
\end{equation}
When $m=2$, we use the fact that $\homclass{v'}{1}{P_{t-1}}{P_{4}} > 0$ to say
\begin{eqnarray*}
\homclass{v}{1}{P_t}{P_{4}} &=& \homclass{v'}{2}{P_{t-1}}{P_{4}} \\
	&<& \homclass{v'}{1}{P_{t-1}}{P_{4}} + \homclass{v'}{2}{P_{t-1}}{P_{4}} \\
	&=& \homclass{v}{2}{P_t}{P_4}.
\end{eqnarray*}
which establishes inequality~\eqref{inq:paths}.

For $m \ge 3$, we first establish $\homclass{v}{1}{P_t}{P_{2m}} < \homclass{v}{2}{P_t}{P_{2m}}$. By inequality~\eqref{inq:path_induction}, we see that $\homclass{v'}{2}{P_{t-1}}{P_{2m}} \le \homclass{v'}{3}{P_{t-1}}{P_{2m}}$. Then, again using $\homclass{v'}{1}{P_{t-1}}{P_{2m}} > 0$, we have
\begin{eqnarray*}
\homclass{v}{1}{P_t}{P_{2m}} &=& \homclass{v'}{2}{P_{t-1}}{P_{2m}} \\
	&\le& \homclass{v'}{3}{P_{t-1}}{P_{2m}} \\
	&<& \homclass{v'}{1}{P_{t-1}}{P_{2m}} + \homclass{v'}{3}{P_{t-1}}{P_{2m}} \\
	&=& \homclass{v}{2}{P_t}{P_{2m}}.
\end{eqnarray*}
Next, we show $\homclass{v}{i}{P_t}{P_{2m}} \le \homclass{v}{i+1}{P_t}{P_{2m}}$ for $2 \le i \le m-2$ by using inequality~\eqref{inq:path_induction} to say
\begin{eqnarray*}
\homclass{v}{i}{P_t}{P_{2m}} &=& \homclass{v'}{i-1}{P_{t-1}}{P_{2m}} + \homclass{v'}{i+1}{P_{t-1}}{P_{2m}} \\
	&\le& \homclass{v'}{i}{P_{t-1}}{P_{2m}} + \homclass{v'}{i+2}{P_{t-1}}{P_{2m}} \\
	&=& \homclass{v}{i+1}{P_t}{P_{2m}}.
\end{eqnarray*}
Finally, we demonstrate $\homclass{v}{m-1}{P_t}{P_{2m}} \le \homclass{v}{m}{P_t}{P_{2m}}$ by once again using inequality~\eqref{inq:path_induction}:
\begin{eqnarray*}
\homclass{v}{m-1}{P_t}{P_{2m}} &=& \homclass{v'}{m-2}{P_{t-1}}{P_{2m}} + \homclass{v'}{m}{P_{t-1}}{P_{2m}} \\
	&\le& \homclass{v'}{m-1}{P_{t-1}}{P_{2m}} + \homclass{v'}{m}{P_{t-1}}{P_{2m}} \\
	&=& \homclass{v}{m}{P_t}{P_{2m}}.
\end{eqnarray*}
which establishes inequality~\eqref{inq:paths}.

Finally, we use inequality~\eqref{inq:paths} to show that $P_{2m}$ is strongly Hoffman-London. We consider the cases $m=2$ and $m > 2$ separately. For $m=2$ and $t \ge 2$, set $a(t)=1$ and $b(t)=2$. From inequality~\eqref{inq:paths}, we have $\homclass{v}{a(t)}{P_s}{P_{4}} < \homclass{v}{b(t)}{P_s}{P_{4}}$ for all $s \ge 2$. To apply \cref{cor:uniquenessalt} to show that $P_{4}$ is strongly Hoffman-London, it remains to show that for all $t \ge 2$ there is a $P_{4}$-coloring of $P_t$ that sends one endvertex $v$ of $P_t$ to some vertex in $H^1$ and the other endvertex $w$ to some vertex in $H^2$. For $t$ even, such a coloring is given by sending $v$ to $u_1\in H^1$, then sending the vertices of $P_t$ alternatively to $u_2$ and $u_1$, ending by sending $w$ to $u_2 \in H^2$. For $t$ odd, such a coloring is given by sending $v$ to $u_1\in H^1$, then sending the vertices of $P_t$ alternatively to $u_2$ and $u_1$ except for $w$, which is sent to $u_3 \in H^2$ instead of $u_1$.

For $m>2$, set $a(t)=1$ for $t \ge 2$, and set $b(t)=2$ for even $t \ge 2$ and set $b(t)=3$ for odd $t \ge 2$. From inequality~\eqref{inq:paths}, we again have $\homclass{v}{a(t)}{P_s}{P_{2m}} < \homclass{v}{b(t)}{P_s}{P_{2m}}$ for all $s \ge 2$. To exhibit a coloring of $P_t$ that sends $v$ to a vertex in $H^{a(t)}$ and $w$ to a vertex in $H^{b(t)}$, we proceed exactly as in the $m=2$ case: send $v$ to $u_1 \in H^1$ and proceed to alternate sending the vertices of $P_t$ to $u_2$ and $u_1$, ending by sending $w$ to $u_2\in H^2$ when $t$ is even and deviating by sending $w$ to $u_3 \in H^3$ instead of $u_1$ when $t$ is odd. 
\end{proof}

\cref{thm:even_path_minimizer} shows that $P_\ell$ is strongly Hoffman-London for all even $\ell \ge 4$,  and \cref{obs:regular} establishes that $P_2$ is not strongly Hoffman-London, because all trees admit the same number of $P_2$-colorings. What can be said of $P_\ell$ when $\ell$ is odd? As demonstrated in \cref{thm:complete_bipartite}, $P_3$ is not strongly Hoffman-London but in a different manner than $P_2$: trees with a balanced bipartition, including paths but not all trees, minimize $\hom(T_n,P_3)$. We know that paths are among the minimizers of $\hom(T_n,P_\ell)$ for odd $\ell$ due to Csikv\'ari and Lin~\cite{CsikvariLin2015}, but the precise set of trees that minimize the number of $P_\ell$-colorings for odd $\ell \ge 5$ remains open.

The {\it fully looped path} $P_n^\loop$ is the graph on vertex set $u_1, \ldots, u_n$ with edges $u_iu_{i+1}$ for $1 \le i \le n-1$ and $u_iu_i$ for each $1 \le i \le n$. Unlike for unlooped paths, for which we were only able to apply our techniques when the number of vertices is even, we can completely classify all fully looped paths.

\begin{theorem} \label{prop:looped_path}
The fully looped path $P_n^\loop$ is Hoffman-London for all $n \ge 1$ and strongly Hoffman-London for all $n \ge 3$.
\end{theorem}

\begin{proof}
For $n = 1$ or $2$, $P_n^\loop$ is a fully looped complete graph and so \cref{obs:regular} tells us that $P_n^\loop$ is Hoffman-London but not strongly Hoffman-London. When $n=3$, $P_n^\loop$ is an instance of a regular graph with a looped dominating vertex added, and so applying \cref{thm:dom_reg} establishes that $P_n^\loop$ is strongly Hoffman-London.

For $n \ge 4$, we start by considering $n$ even, say $n=2m$. Exactly as for $P_{2m}$, we find that the automorphic similarity classes of $P_{2m}^\circ$ are $H^i=\{v_i,v_{2m+1-i}\}$, $i=1, \ldots, m$, each of size $2$, and the the associated automorphic similarity matrix has the form 
\[ \begin{bmatrix}
1 & 1 & 0 & 0 & \cdots & 0 & 0 & 0\\
1 & 1 & 1 & 0 & \cdots & 0 & 0 & 0\\
0 & 1 & 1 & 1 & \cdots & 0 & 0 & 0\\
\vdots & \vdots & \vdots & \vdots & \ddots & \vdots & \vdots& \vdots\\
0 & 0 & 0 & 0 & \cdots & 1 & 1 & 1\\
0 & 0 & 0 & 0 & \cdots & 0 & 1 & 2
    \end{bmatrix} \]
(i.e., it is the automorphic similarity matrix of $P_{2m}$ modified by the addition of the identity matrix $I_{2m}$). It is straightforward to check that all terminal sums of the \ith{i} row are less than or equal to the corresponding terminal sums of the \ith{(i+1)} row for all $i < m$, and so this matrix satisfies the increasing columns property and $P_{2m}^\circ$ is Hoffman-London. 

For $n \ge 5$ and odd, say $n=2m-1$, the automorphic similarity classes of $P_{2m-1}^\loop$ are $H^i=\{v_i,v_{2m-i}\}$, $i=1, \ldots, m-1$, each of size $2$, and $H^m=\{v_m\}$ (of size $1$). The structure of the corresponding automorphic similarity matrix is the same as in the even case, except that in the final row the $1$ and $2$ are flipped. That $M$ satisfies the increasing columns condition follows almost exactly as it did in the even case.

We have established that $P_n^\loop$ is Hoffman-London for all $n$ and strongly Hoffman-London for $n = 3$. To establish that $P_n^\loop$ is strongly Hoffman-London for $n \ge 4$, we will show that \cref{cor:uniquenessalt} can be applied. We first observe that for all $s \ge 2$ we have $\homclass{v}{1}{P_s}{P_n^\loop} < \homclass{v}{2}{P_s}{P_n^\loop}$, where $v$ is one of the endvertices of $P_s$. We can see this via a non-surjective injection from $\Homclass{v}{1}{P_s}{P_n^\loop}$ into $\Homclass{v}{2}{P_s}{P_n^\loop}$. For concreteness, we take a fixed $x \in H^1$ (necessarily one of the endvertices of $P_n^\loop$) to be the representative of $H^1$ that $v$ is mapped to, and take its unique neighbor $x' \in H^2$ to be the representative of $H^2$ that $v$ is mapped to. The injection is simple: given $f \in \Homclass{v}{1}{P_s}{P_n^\loop}$ with $f(v)=x$, map $f$ to $f'$ by changing the value at $v$ to $x'$. That $f' \in \Homclass{v}{2}{P_s}{P_n^\loop}$ (i.e., that $f'$ is actually a homomorphism) follows from the fact that $f'$ maps $v'$ (the unique neighbor of $v$ in $P_s$) to a neighbor of $x$, so either $x$ or $x'$, and both of these are neighbors of $x'$. To see that the injection is not a surjection, consider $f' \in \Homclass{v}{2}{P_s}{P_n^\loop}$ that sends $v$ to $x'$ and all other vertices of $P_s$ to $x''$, the neighbor of $x'$ in $H^2$. There is no $f \in \Homclass{v}{1}{P_s}{P_n^\loop}$ that maps onto $f'$ via our injection, since $x$ is not adjacent to $x''$ in $P_n^\loop$.

We next observe that for all $t \ge 2$ there is a $P_n^\loop$-coloring of $P_t$ that sends one endvertex to a vertex in $H^1$ and the other to a vertex in $H^2$---simply send one endvertex to $x$ and all the others to $x'$. That $P_n^\loop$ ($n \ge 4$) is strongly Hoffman-London now follows from \cref{cor:uniquenessalt}.
\end{proof}

We end our discussion of paths with a quick corollary of \cref{prop:looped_path} to the family of {\it ladder graphs}. The ladder graph $L_n$ has vertices $u_1, \ldots, u_n, v_1, \ldots, v_n$ and edges $u_iu_{i+1}$ and $v_iv_{i+1}$, for $1 \le i \le n-1$, and $u_iv_i$ for $1 \le i \le n$. 
\begin{corollary}
For $n \ge 1$, the ladder graph $L_n$ is Hoffman-London. For $n \ge 3$ it is strongly Hoffman-London.
\end{corollary}

\begin{proof}
We have $L_n = P_n^\loop \times K_2$, so the proposition follows from \cref{obs:cartesian_2} and \cref{prop:looped_path}. 
\end{proof}

\subsection{Loop threshold graphs} \label{sec:loop-thresh}

Here we define the family of loop threshold graphs and prove that they are Hoffman-London. Denote by $K_1$ the graph on one vertex with no edges, and by $K^{\loop}_1$ the graph on one vertex with a loop. The family of \emph{loop threshold graphs} is the minimal (by inclusion) family of graphs that includes $K_1$ and $K^{\loop}_1$, is closed under adding isolated vertices, and is closed under adding looped dominating vertices. Loop threshold graphs were introduced by Cutler and Radcliffe~\cite{CutlerRadcliffe2014} where they were shown to be a natural family to consider in the context of $H$-coloring; one of the results of \cite{CutlerRadcliffe2014} is that if $H$ is loop threshold, then among $n$-vertex $m$-edge graphs, there is a threshold graph that maximizes the number of $H$-colorings. For further work on loop threshold graphs, see also~\cite{CutlerKass2020} (regarding $H$-coloring) and~\cite{GalvinWesleyZacovic2022} (regarding enumeration). 

The following characterization of loop threshold graphs is stated in~\cite{CutlerRadcliffe2014}, where the authors suggest adapting a proof from~\cite{ChvatalHammer1977}. For completeness, we present the details here.

\begin{lemma} \label{lem:loop_thresh_char}
For a graph $H$ on $n$ vertices, possibly with loops, the following are equivalent.
\begin{enumerate}
\item $H$ is loop threshold.
\item There is an ordering of the vertices of $H$, say $v_1, v_2, \ldots, v_n$, with the property that
\[N(v_1) \subseteq N(v_2) \subseteq \cdots \subseteq N(v_n).\]
\end{enumerate}
\end{lemma}
We refer to an ordering of vertices satisfying the second condition in \cref{lem:loop_thresh_char} as a \emph{nested neighborhood ordering}.
 
\begin{proof}
We begin by showing that if $H$ is a loop threshold graph, then it admits a nested neighborhood ordering. Let ${\mathcal O}$ be the family of graphs for which a nested neighborhood ordering exists. Trivially, $K_1, K^{\loop}_1 \in {\mathcal O}$. Suppose now that $H \in {\mathcal O}$. Let $v_1, \ldots, v_n$ be a nested neighborhood ordering for $H$. If $H'$ is obtained from $H$ by adding an isolated vertex, $u$, then since $N_{H'}(u)=\varnothing$ we have that $u, v_1, \ldots, v_{n}$ is a nested neighborhood ordering for $H'$, so $H' \in {\mathcal O}$. If instead $H'$ is obtained from $H$ by adding a looped dominating vertex, $u$, then since $N_{H'}(u)=\{v_1, \ldots, v_{n}, u\}$, we have that $v_1, \ldots, v_{n}, u$ is a nested neighborhood ordering for $H'$, so $H' \in {\mathcal O}$. It follows from the definition of the family of loop threshold graphs that all loop threshold graphs are in ${\mathcal O}$.    

For the reverse implication, let $H$ be a graph on $n$ vertices that admits a nested neighborhood ordering $v_1, \ldots, v_n$. We show by induction on $n$ that $H$ is loop threshold. When $n=1$ this is clear, so we assume $n > 1$. 

If $N(v_n)=\{v_1, \ldots, v_n\}$, then $v_n$ is a looped dominating vertex in $H$. Let $H'$ be obtained from $H$ by deleting $v_n$. We have that
\[
N(v_1)\setminus \{v_n\} \subseteq N(v_2)\setminus \{v_n\} \subseteq \cdots N(v_{n-1})\setminus \{v_n\}
\]
and so $v_1, \ldots, v_{n-1}$ is a nested neighborhood ordering for $H'$. By induction, $H'$ is a loop threshold graph, so $H$, being obtained from $H'$ by adding a looped dominating vertex, is also loop threshold.

Otherwise there is $v_i$ in $H$ that is not in $N(v_n)$. Then $v_i$ must be isolated: if there were $v_j$ such that $v_i \sim v_j$, then $v_i \in N(v_j) \subseteq N(v_n)$. So we observe that for each $j$, $N(v_j) \setminus \{v_i\} = N(v_j)$, meaning if $H'$ is obtained from $H$ by deleting $v_i$, then $H'$ admits the nested neighborhood ordering $v_1, \ldots, v_{i-1}, v_{i+1}, \ldots, v_n$. By induction, $H'$ is a loop threshold graph, and as we obtain $H$ by adding $v_i$, an isolated vertex, $H$ is also loop threshold.
\end{proof}

Using \cref{lem:loop_thresh_char}, we now prove that loop threshold graphs are Hoffman-London, and we characterize the strongly Hoffman-London loop threshold graphs.

\begin{theorem}\label{thm:loop_threshold}
Each loop threshold graph is Hoffman-London. Furthermore, if $H$ is a loop threshold graph, then $H$ is strongly Hoffman-London if and only after removing all isolated vertices from $H$, the resulting graph is neither empty nor a fully looped complete graph.
\end{theorem}

Note that we consider a vertex $v$ isolated if $N(v) \subseteq \{v\}$. In other words, a single vertex with a loop is still an isolated vertex.

\begin{proof} 
Let $H$ be a loop threshold graph. To establish the Hoffman-London property, we seek to apply \cref{thm:ordering_framework}. By \cref{lem:loop_thresh_char}, there is an ordering of the vertices of $H$, say $v_1, v_2, \ldots, v_n$, with the property that
\[N(v_1) \subseteq N(v_2) \subseteq \cdots \subseteq N(v_n).\]
Suppose $v_i \in H^a$ and $v_{i+1} \in H^b$ for some $a \ne b$. We claim $v_j \notin H^a$ for all $j > i$. Suppose otherwise: assume $v_i \in H^a$, $v_{i+1} \in H^b$, and $v_j \in H^a$ for some $j > i$. As $v_a$ and $v_j$ both belong to $H^a$, $|N(v_i)| = |N(v_j)|$, and as $i < j$, $N(v_i) \subseteq N(v_j)$, so $N(v_i) = N(v_j)$. Then
\[ N(v_j) = N(v_i) \subseteq N(v_{i+1}) \subseteq N(v_j) \]
so $N(v_{i+1}) = N(v_i)$ as well. Thus the function on $V(H)$ that swaps $v_i$ and $v_{i+1}$ and fixes all other vertices is an automorphism of $H$, contradicting that $v_{i+1} \notin H^a$. We conclude that this ordering of the vertices induces an ordering of the automorphic similarity equivalence classes of $H$, say $H^1, \ldots, H^k$, such that if $u \in H^i$ and $v \in H^j$ for $i < j$ then $N(u) \subseteq N(v)$. If we use this ordering of the automorphic similarity classes to define the automorphic similarity matrix $M$, we see that each column of $M$ is non-decreasing: for any $1 \le j \le k$ and $1 \le i \le k-1$, $m_{i,j} \le m_{i+1,j}$ because for any $u \in H^i$ and $v \in H^{i+1}$, $N(u) \subseteq N(v)$, so each $w \in H^j$ is adjacent to $u$ only if it is adjacent to $v$. Sums of non-decreasing vectors are also non-decreasing, so this ordering of the automorphic similarity classes of $H$ satisfies the hypothesis of \cref{thm:ordering_framework}. We conclude that $H$ is Hoffman-London, as claimed.

We now turn to the characterization of loop threshold graphs that are strongly Hoffman-London. For $n \ge 2$, each vertex of any tree on $n$ vertices is incident to an edge and cannot be mapped to an isolated vertex of $H$. (This is essentially the content of \cref{obs:isolated}). Thus if $H$ is loop threshold and $H'$ is the graph resulting from removing the isolated vertices from $H$, we have $\hom(T_n,H) = \hom(T_n,H')$ for all $n \ge 2$. If $H'$ is empty, then $\hom(T_n,H') = 0$ for all $T_n \in {\mathcal T}_n$. If $H'$ is a fully looped complete graph, it is regular, so each $T_n \in {\mathcal T}_n$ admits the same number of $H$-colorings by \cref{obs:regular}.

Now assume that upon deletion of all isolated vertices, $H$ is non-empty and not a fully looped complete graph. Note that the resulting graph is also loop threshold: because only isolated vertices were removed, the neighborhoods of the remaining vertices are unchanged, so the resulting graph still admits a nested neighborhood ordering. Thus it suffices to assume $H$ simply has no isolated vertices and is not a fully looped complete graph. To show that $H$ is strongly Hoffman-London, we use \cref{cor:uniquenessalt}.

We begin by showing that $H$ has at least two automorphic similarity classes. If every vertex of $H$ is looped, then $H$ must have been constructed by adding looped dominating vertices at every step, contradicting that $H$ is not a fully looped complete graph. Let $u$ be a vertex of $H$ without a loop. Because $u$ was not removed, it is not isolated, so there is $v \sim u$. We claim that $u$ comes before $v$ in the nested neighborhood ordering of $H$; if not, then $u \in N(v) \subseteq N(u)$, which contradicts our claim that $u$ does not have a loop. By similar reasoning, $v$ must have a loop: $v \in N(u) \subseteq N(v)$. Because $u$ does not have a loop but $v$ does, no automorphism of $H$ can send $u$ to $v$. We conclude $H$ has at least two distinct automorphic similarity classes.

Now for each $t \ge 2$, let $a(t)=i$, where $H^i$ is the class containing $u$, and let $b(t)=j$, where $H^j$ is the class containing $v$. We first observe that there is an $H$-coloring of $P_t$ that sends the first vertex of the path to some vertex in $H^{a(t)}$ and the last vertex of the path to some vertex in $H^{b(t)}$. Indeed, mapping the first vertex of $P_t$ to $u$ and all subsequent vertices to $v$ yields such a coloring.

To apply \cref{cor:uniquenessalt} to conclude that $H$ is strongly Hoffman-London, it remains to show that for each $t \ge 2$ and $s \ge 2$, $\homclass{w}{b(t)}{P_s}{H} > \homclass{w}{a(t)}{P_s}{H}$ where $w$ is one of the endvertices of $P_s$. In what follows, when enumerating $\homclass{w}{a(t)}{P_s}{H}$ (respectively, $\homclass{w}{b(t)}{P_s}{H}$) we will take $u$ (respectively, $v$) to be the specific vertex of $H^{a(t)}$ (respectively, $H^{b(t)}$) to which $w$ is mapped in an $H$-coloring. There is a simple injection from $\Homclass{w}{a(t)}{P_s}{H}$ into $\Homclass{w}{b(t)}{P_s}{H}$: change the color at $w$ from $u$ to $v$. That this is a valid map comes from the fact that $N(u) \subseteq N(v)$. To see that this injection is not a bijection, and so $\homclass{w}{b(t)}{P_s}{H} > \homclass{w}{a(t)}{P_s}{H}$, consider the $H$-coloring of $P_s$ that sends $w$ to $v$, that sends the unique neighbor of $w$ in $P_s$ to $u$, and that sends all other vertices of $P_s$ (if there are any) to $v$. This coloring is in $\Homclass{w}{b(t)}{P_s}{H}$, but, since there is not a loop at $u$ in $H$, it is not the image of any coloring in $\Homclass{w}{a(t)}{P_s}{H}$.
\end{proof}

We conclude our discussion of loop threshold graphs with an application of \cref{thm:loop_threshold} to a family of target graphs that have arisen both in statistical physics and communications networks. For integer $C \ge 1$, define a \emph{capacity $C$ independent set} in a graph $G$ to be a function $f:V(G) \to \{0,1,2,\ldots, C\}$ satisfying $f(x)+f(y)\le C$ for all $xy \in E(G)$. Standard independent sets are capacity 1 independent sets. This notion was introduced by Mazel and Suhov~\cite{MazelSuhov1991} as a statistical physics model of a random surface and was subsequently studied in~\cite{GalvinMartinelliRamananTetali2011,RamananSenguptaZiedinsMitra2002} in the context of multicast communications networks. Capacity $C$ independent sets can be encoded as $H$-colorings by taking $H=H_C$ to be the graph on vertex set $\{0,1,\ldots, C\}$ with edge set $\{\{a,b\}: a+b \le C\}$. The graph $H_C$ admits a nested neighborhood ordering, namely $C, C-1, \ldots, 1, 0$, since $N_{H_C}(i)=\{0,1,\ldots, C-i\}$, so by \cref{lem:loop_thresh_char} $H_C$ is loop threshold, and the following result is a corollary of \cref{thm:loop_threshold}.

\begin{corollary}
For $C \ge 1$ and $n \ge 1$, if $T_n$ is a tree on $n$ vertices that is different from $P_n$, then $P_n$ admits strictly fewer capacity $C$ independent sets than $T_n$.   
\end{corollary}

\subsection{Adding looped dominating vertices to a target graph}

In this section we establish that the set of graphs $H$ for which \cref{thm:ordering_framework} is applicable is closed under adding looped dominating vertices and present an application that will be useful later.

\begin{proposition} \label{prop:closure}
Suppose that $H$ is a target graph that has an automorphic similarity matrix satisfying the increasing columns property. Let $\widetilde{H}$ be obtained from $H$ by adding some number of looped dominating vertices. Then $\widetilde{H}$ is Hoffman-London.
\end{proposition}

\begin{proof}
We begin by considering the case where $H$ does not have any looped dominating vertices. Let the automorphic similarity classes of $H$ be $H^1, \ldots, H^k$, presented in such an order that the associated automorphic similarity matrix $M$ satisfies the increasing columns property. Then let $\widetilde{H}$ be obtained from $H$ by adding $b$ looped dominating vertices.

Any automorphism of $\widetilde{H}$ must be an extension of an automorphism of $H$ obtained by permuting the looped dominating vertices of $\widetilde{H}$, and so $\widetilde{H}$ has automorphic similarity classes $\widetilde{H}^1, \ldots, \widetilde{H}^k, \widetilde{H}^{k+1}$ where $\widetilde{H}^i = H^i$ for $i \le k$ and $\widetilde{H}^{k+1}$ is the set of $b$ added looped dominating vertices. The associated automorphic similarity matrix $\widetilde{M}$ is obtained from $M$ by adding a new row at the bottom whose \ith{j} entry ($j=1, \ldots, k$) is $|H^j|$, and then adding a new column at the far left, all of whose entries are $b$. The last column of this matrix is constant and so non-decreasing. 

For $i$ and $j$ with $1 \le i \le k-1$ and $1 \le j \le k$, the sum $\widetilde{S}^{i,j}$ of the entries of row $i$ of $\widetilde{M}$ starting from the \ith{j} entry is $b$ greater than $S^{i,j}$, the sum of the entries of row $i$ of $M$ starting from the \ith{j} entry, and similarly $\widetilde{S}^{i+1,j} = S^{i+1,j}+b$, so $\widetilde{S}^{i,j} \le \widetilde{S}^{i+1,j}$ follows from $S^{i,j} \le S^{i+1,j}$. 

Finally, the $(k,j)$-entry ($1 \le j \le k$) of $\widetilde{M}$, which counts the number of neighbors an element of $H^k$ has in $H^j$, is trivially bounded above by $|H^j|$, which is the $(k+1,j)$-entry of $\widetilde{M}$; this, together with the fact that the \ith{(k,k+1)} and \ith{(k+1,k+1)} entries of $\widetilde{M}$ agree, shows that $\widetilde{S}^{k,j} \le \widetilde{S}^{k+1,j}$. We conclude that $\widetilde{M}$ has the increasing columns property and that $\widetilde{H}$ is Hoffman-London.

We now consider the case where $H$ has some looped dominating vertices. The collection of looped dominating vertices in $H$ forms an autmorphic similarity class. Moreover if the automorphic similarity classes of $H$ are $H^1, \ldots, H^k$, presented in such an order that the associated automorphic similarity matrix $M$ satisfies the increasing columns property, then the class $H^k$ must be the class consisting of all the looped dominating vertices. To see this, note that the increasing columns property applied to the sum of all $k$ columns of $M$ simply says $d_1 \le d_2 \le \cdots \le d_k$, where $d_i$ is the common degree of the vertices in $H^i$. Looped dominating vertices have the maximum possible degree, namely $|H|$, and any vertex with degree $|H|$ is a looped dominating vertices. It follows that $d_k=|H|$ and that $H^k$ consists of looped dominating vertices.

Now once again let $\widetilde{H}$ be obtained from $H$ by adding $b$ looped dominating vertices. The automorphic similarity classes of $\widetilde{H}$ are $\widetilde{H}^1, \widetilde{H}^2, \ldots, \widetilde{H}^k$, where $\widetilde{H}^i=H^i$ for $i < k$ and $\widetilde{H}^k$ is the union of $H^k$ and the $b$ added looped dominating vertices. The associated automorphic similarity matrix $\widetilde{M}$ is obtained from $M$ by adding $b$ to every entry in the final column of $M$. In a similar manner to the previous case, it follows that $\widetilde{M}$ has the increasing columns property and that $\widetilde{H}$ is Hoffman-London.
\end{proof}

Before presenting an application of \cref{prop:closure}, we discuss blow-ups of graphs. 
\begin{definition}
Let $H$ be a target graph on vertex set $\{v_0, v_1, \ldots, v_k\}$ without multiple edges but perhaps with loops. Let ${\bf n}=(n_0, n_1, \ldots, n_k)$ be a vector of positive integers. The {\it blow-up of $H$ by ${\bf n}$}, denoted $H^{\bf n}$, has vertex set $\bigcup_{i=0}^k V_i$, where $V_i$ is a set of size $n_i$ and the $V_i$s are pairwise disjoint. If $v_i \sim_H v_j$ (with either $i \ne j$ or $i=j$), then there is an edge in $H^{\bf n}$ between every $x \in V_i$ and $y \in V_j$, and there are no other edges. 
\end{definition}
\noindent In other words, $H^{\bf n}$ is obtained from $H$ by ``blowing up'' each vertex of $H$ to a cluster of vertices, replacing unlooped vertices with empty graphs, looped vertices with fully looped complete graphs, and edges with complete bipartite graphs. As an example, the join of an empty graph and a fully looped complete graph is a blow-up of $H_{\rm ind}$, the target graph that encodes independent sets; see \cref{fig:blowup}.

\begin{figure}[ht]
    \centering
    \captionsetup{width=.7\linewidth}
    \begin{tikzpicture}
        \coordinate (v) at (0,0);
        \coordinate (w) at (2,0);

        \draw[fill=black] (v) circle (3pt);
        \draw[fill=black] (w) circle (3pt);

        \draw[thick,black] (v) -- (w);
        \draw[thick,black] (w) to [out=-50,in=50,loop,style={min distance=10mm}] (w);

        \node at (3,0) {};

        \node[label=below:{$H$}] at (1,-1.08) {};
    \end{tikzpicture}
    \begin{tikzpicture}
        \node at (-1.5,0) {};
        \coordinate (v1) at (-.5,0);
        \coordinate (v2) at (0.25,.43);
        \coordinate (v3) at (0.25,-.43);

        \coordinate(w1) at (2,0.5);
        \coordinate(w2) at (2,-0.5);

        \draw[fill=black] (v1) circle (3pt);
        \draw[fill=black] (v2) circle (3pt);
        \draw[fill=black] (v3) circle (3pt);
        \draw[fill=black] (w1) circle (3pt);
        \draw[fill=black] (w2) circle (3pt);

        \draw[thick,black] (w1) to [out=-5,in=95,loop,style={min distance=10mm}] (w1);
        \draw[thick,black] (w2) to [out=-5,in=-95,loop,style={min distance=10mm}] (w2);
        \draw[thick,black] (w1) -- (w2) -- (v1) -- (w1) -- (v2) -- (w2) -- (v3) -- (w1);

        \node[label=below:{$H^{\bf n}$}] at (1,-1) {};
    \end{tikzpicture}
    \caption{An example of a blow-up with ${\bf n} = (3,2)$ where the unlooped vertex of $H = H_{\ind}$ is labeled $v_0$ and the looped vertex is labeled $v_1$.} \label{fig:blowup}
\end{figure}
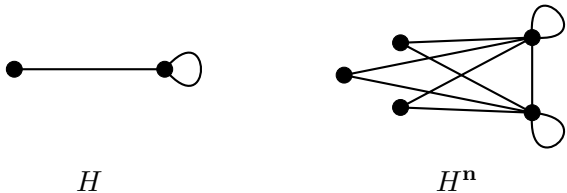

We now show, as a corollary of \cref{prop:closure}, that blow-ups of fully looped stars (stars augmented by putting a loop at every vertex) are Hoffman-London. We will present an application of this corollary, to the Widom-Rowlinson model from statistical physics, in \cref{sec:stat-phys}.

\begin{corollary} \label{cor:blowup-of-fully-looped-star}
Let $S_{t+1}^\circ$ be the fully looped star on $t+1$ vertices $v_0, v_1, \ldots, v_t$, with $v_0$ the center of the star. Let ${\bf n}=(n_0,n_1, \ldots, n_t)$ be any vector of positive integers. The blow-up of $S_{t+1}^\circ$ by ${\bf n}$ is Hoffman-London.
\end{corollary}

\begin{proof}
We start by considering the graph $H(n_1, \ldots, n_t)$, which is the disjoint union of $t$ fully looped complete graphs $K^\loop_1, \ldots, K^\loop_t$ on $n_1, \ldots, n_t$ vertices, respectively. Let $S_1, \ldots, S_k$ be the partition of $\{1, \ldots, t\}$ induced by the equivalence relation $i \equiv j$ if and only if $n_i=n_j$, and let $s_i=\sum_{j \in S_i} n_j$. Without loss of generality, we may assume that $s_1 \le s_2 \le \cdots \le s_k$. For example, if $(n_1, n_2, n_3, n_4, n_5, n_6)=(3,10,4,3,10,10)$, then we would take $S_1=\{3\}$ with $s_1=4$, $S_2=\{1,4\}$ with $s_2=3+3=6$ and $S_3=\{2,5,6\}$ with $s_3=10+10+10=30$. The automorphic similarity classes of $H$ are $H^i=\bigcup_{j \in S_i} V(K^j)$ and the corresponding automorphic similarity matrix is $\diag\{s_1, \ldots, s_k\}$. By our choice of ordering on the $s_i$, this satisfies the increasing columns property. 

It follows from \cref{thm:ordering_framework} that $H(n_1, \ldots, n_t)$ is Hoffman-London, and therefore it follows from \cref{prop:closure} that the graph $H(n_0, n_1, \ldots, n_t)$ obtained from $H(n_1, \ldots, n_t)$ by adding $n_0$ looped dominating vertices is also Hoffman-London. But $H(n_0,n_1, \ldots, n_t)$ is exactly the blow-up of $S_{t+1}^\circ$ by ${\bf n}$, and this observation completes the proof. 
\end{proof}

\subsection{Connection to statistical physics} \label{sec:stat-phys}

In this section we briefly discuss a connection between graph homomorphisms and statistical physics models, and we use some of our results to derive consequences for the partition functions of some of these models. For more thorough (though still accessible) treatments of the connection, see e.g.~\cite{BrightwellWinkler1999, BrightwellWinkler2002, delaHarpeJones1993}.

One goal of statistical physics is to understand the global behavior of models that are defined by local rules. In a {\it hard constraint spin model}, space is modeled by a graph $G$. The vertices of $G$ represent sites that are occupied by at most one particle per site. Each particle has a certain spin, from among some specified collection of possible spins. The edges of $G$ encode pairs of sites that are close enough that the particles at those sites interact. The interaction rule in a hard constraint model is that only certain pairs of spins are allowed to appear on adjacent sites. This can be encoded by a graph $H$ whose vertices are the possible spins, with an edge between two vertices exactly when the corresponding spins are allowed to appear on adjacent sites. A valid configuration of spins on the vertices of $G$ is thus exactly an $H$-coloring of $G$. Note that in this context it is very natural to allow $H$ to have loops---a loop at vertex $v \in V(H)$ encodes the fact that two particles on adjacent sites of $G$ are allowed to both have spin $v$.

Typically, a hard-constraint spin model, with constraint graph $H$ on vertex set $\{v_0, \ldots, v_k\}$, comes with a vector of positive activities ${\boldsymbol\lambda}=(\lambda_0, \lambda_1, \ldots, \lambda_k)$ which (informally) measure how frequently particles of each spin typically occur. Here $\lambda_i$ is the activity associated with $v_i$. The weight of a configuration, or equivalently the weight of an $H$-coloring $f:V(G)\rightarrow V(H)$, is given by 
\begin{equation} \label{eq:config-weight}
w^{(H,{\boldsymbol\lambda})}(f) = \prod_{v \in V(G)} \lambda_{f(v)} = \prod_{i=0}^k \lambda_i^{|f^{-1}(v_i)|}
\end{equation}
where note that $f^{-1}(v_i)$ is the set of vertices in $G$ that are mapped to $v_i$ by $f$. The {\it partition function} of the model is the sum 
\[
Z^{\boldsymbol\lambda}(G,H) = \sum_{f \in \
Hom(G,H)} w^{(H,{\boldsymbol\lambda})}(f).
\]
The partition function is the normalizing constant used to turn the assignment of weights into an assignment of probabilities, making the model stochastic. The partition function encodes significant information about the model.

We present three specific examples; these are the three for which we later present results.
\begin{example} \label{ex:hc-model}
The {\it hard-core} or {\it lattice gas model} is a model of the occupation of space by atoms that have non-zero volume, with a ``hard core'' around their boundary, meaning that if an atom sits at a site, no other atom can sit at any nearby site. Dating back at least to the investigations of Gaunt and Fisher and of Runnels \cite{GauntFisher1965, Runnels1965}, this model can be represented using two spins, $v_{\rm out}$ representing ``unoccupied'' and $v_{\rm in}$ representing ``occupied,'' with the occupation rule being that unoccupied sites can be adjacent to each other, unoccupied sites can be adjacent to occupied sites, and occupied sites cannot be adjacent to each other. Since possible collections of occupied sites in this model correspond exactly with the collection of independent sets in $G$, the hard-core model can be encoded using the target graph $H_{\rm ind}$, where $v_{\rm out}$ is the looped vertex of $H_{\rm ind}$ and $v_{\rm in}$ is the unlooped vertex. Traditionally, the activity associated with $v_{\rm out}$ in this model is $1$ and the activity associated with $v_{\rm in}$ is some $\lambda > 0$. If $\lambda$ is small, then the model favors sparser configurations of occupied vertices, while if $\lambda$ is large, then it favors denser configurations. If $\lambda=1$, then the model gives all configurations equal weight. See \cref{fig:hardcore} for an example.
\end{example}

\begin{figure}[ht]
    \centering
    \captionsetup{width=.7\linewidth}
    \begin{tikzpicture}
        \coordinate[label=below:{$v_{\rm out}$}] (vo) at (0,0);
        \coordinate[label=below:{$v_{\rm in}$}] (vi) at (2,0);

        \draw[thick,black] (vo) -- (vi);
        \draw[thick,black] (vo) to [out=40,in=140,loop,style={min distance=10mm}] (vo);

        \draw[black,fill=white] (vo) circle (3pt);
        \draw[fill=black] (vi) circle (3pt);

        \node at (2.5,-1.75) {};
    \end{tikzpicture}
    \hspace{1mm}
    \begin{tikzpicture}
        \coordinate (v1) at (-0.5,0.87);
        \coordinate (v2) at (-0.5,-0.87);
        \coordinate (v3) at (0,0);
        \coordinate (v4) at (1,0);
        \coordinate (v5) at (1,1.75);
        \coordinate (v6) at (1,-1.75);
        \coordinate (v7) at (1.5,0.87);
        \coordinate (v8) at (1.5,-0.87);
        \coordinate (v9) at (2.5,0.87);
        \coordinate (v10) at (2.5,-0.87);

        \draw[thick,black] (v1) -- (v3) -- (v2);
        \draw[thick,black] (v3) -- (v4) -- (v7) -- (v9);
        \draw[thick,black] (v4) -- (v8) -- (v10);
        \draw[thick,black] (v7) -- (v5);
        \draw[thick,black] (v8) -- (v6);

        \draw[black,fill=black] (v1) circle (3pt);
        \draw[black,fill=black] (v2) circle (3pt);
        \draw[black,fill=white] (v3) circle (3pt);
        \draw[black,fill=white] (v4) circle (3pt);
        \draw[black,fill=white] (v5) circle (3pt);
        \draw[black,fill=white] (v6) circle (3pt);
        \draw[black,fill=white] (v7) circle (3pt);
        \draw[black,fill=black] (v8) circle (3pt);
        \draw[black,fill=black] (v9) circle (3pt);
        \draw[black,fill=white] (v10) circle (3pt);
        \node at (-1,0) {};
    \end{tikzpicture}
    \caption{The constraint graph encoding the hard-core model (left) and a hard-core configuration on a tree (right). If $v_{\rm out}$ has activity $1$ and $v_{\rm in}$ has activity $\lambda$, then the weight of the illustrated configuration is $\lambda^4$.} \label{fig:hardcore}
\end{figure}
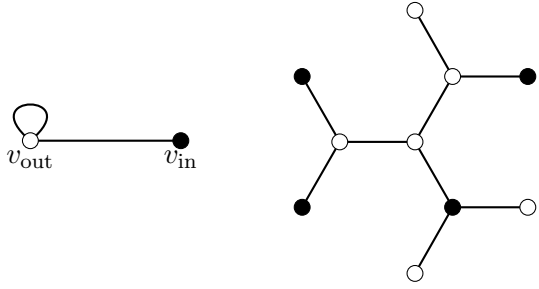

\begin{example}
The {\it Widom-Rowlinson model} is a model of the occupation of space by $k$ mutually repulsive particles. This model was introduced by Widom and Rowlinson~\cite{WidomRowlinson1970} as a model of liquid-vapor phase transition. There are $k+1$ spins, $v_0$ through $v_k$. Empty space is represented by $v_0$, and the remaining spins are used to represent $k$ different particles. A site occupied by a particle of type $i$ can be adjacent to empty space or to other particles of type $i$, but not to particles of type $j$ for any $j \ne i$. Empty space comes with no restriction. A valid configuration of particles is thus modeled by an $H$-coloring of $G$ where $H$ is the fully looped star on vertex set $\{v_0, \ldots v_k\}$ with $v_0$ as the central vertex. A vertex of $G$ being mapped to $v_0$ represents ``unoccupied" and being mapped to $v_i$ represents ``occupied by a particle of type $i$.'' Traditionally, the activity associated with $v_0$ in this model is $1$ and for each $i > 0$ the activity associated with $v_i$ is some $\lambda_i > 0$. See \cref{fig:widomrowlinson} for an example with two particles, denoted red and blue.  
\end{example}

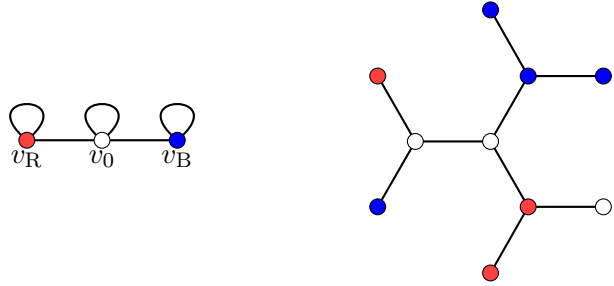
\begin{figure}[ht]
    \centering
    \captionsetup{width=.7\linewidth}
    \begin{tikzpicture}
        \coordinate[label=below:{$v_{\rm R}$}] (vr) at (0,0);
        \coordinate[label=below:{$v_0$}] (v0) at (1,0);
        \coordinate[label=below:{$v_{\rm B}$}] (vb) at (2,0);

        \draw[thick,black] (vr) -- (v0) -- (vb);
        \draw[thick,black] (v0) to [out=40,in=140,loop,style={min distance=10mm}] (v0);
        \draw[thick,black] (vr) to [out=40,in=140,loop,style={min distance=10mm}] (vr);
        \draw[thick,black] (vb) to [out=40,in=140,loop,style={min distance=10mm}] (vb);

        \draw[black,fill=white] (v0) circle (3pt);
        \draw[black,fill=red!75] (vr) circle (3pt);
        \draw[black,fill=blue] (vb) circle (3pt);
        \node at (2.5,-1.75) {};
    \end{tikzpicture}
    \hspace{1cm}
    \begin{tikzpicture}
        \coordinate (v1) at (-0.5,0.87);
        \coordinate (v2) at (-0.5,-0.87);
        \coordinate (v3) at (0,0);
        \coordinate (v4) at (1,0);
        \coordinate (v5) at (1,1.75);
        \coordinate (v6) at (1,-1.75);
        \coordinate (v7) at (1.5,0.87);
        \coordinate (v8) at (1.5,-0.87);
        \coordinate (v9) at (2.5,0.87);
        \coordinate (v10) at (2.5,-0.87);

        \draw[thick,black] (v1) -- (v3) -- (v2);
        \draw[thick,black] (v3) -- (v4) -- (v7) -- (v9);
        \draw[thick,black] (v4) -- (v8) -- (v10);
        \draw[thick,black] (v7) -- (v5);
        \draw[thick,black] (v8) -- (v6);

        \draw[black,fill=red!75] (v1) circle (3pt);
        \draw[black,fill=blue] (v2) circle (3pt);
        \draw[black,fill=white] (v3) circle (3pt);
        \draw[black,fill=white] (v4) circle (3pt);
        \draw[black,fill=blue] (v5) circle (3pt);
        \draw[black,fill=red!75] (v6) circle (3pt);
        \draw[black,fill=blue] (v7) circle (3pt);
        \draw[black,fill=red!75] (v8) circle (3pt);
        \draw[black,fill=blue] (v9) circle (3pt);
        \draw[black,fill=white] (v10) circle (3pt);
        \node at (-1,0) {};
    \end{tikzpicture}
    \caption{The constraint graph encoding the Widom-Rowlinson model (left) and a Widom-Rowlinson configuration on a tree (right). If $v_0$ has activity $1$, $v_{\rm R}$ has activity $\lambda_{\rm R}$, and $v_{\rm B}$ has activity $\lambda_{\rm B}$, then the weight of the illustrated configuration is $\lambda_{\rm R}^3 \lambda_{\rm B}^4$.} 
    \label{fig:widomrowlinson}
\end{figure}

\begin{example}
The \emph{capacity $C$ independent set} or \emph{capacity $C$ multicast model} was introduced in \cref{sec:loop-thresh}. Recall that configurations in this model are encoded as $H$-colorings by taking $H=H_C$ to be the graph on vertex set $\{0,1,\ldots, C\}$ with edge set $\{\{a,b\}: a+b \le C\}$. At $C=1$, $H_C$ reduces to the hard-core model of \cref{ex:hc-model}. Traditionally, in this model the activity assigned to vertex $i$ is $\lambda^i$ for some $\lambda >0$. See \cref{fig:cappacityc} for an example.
\end{example}

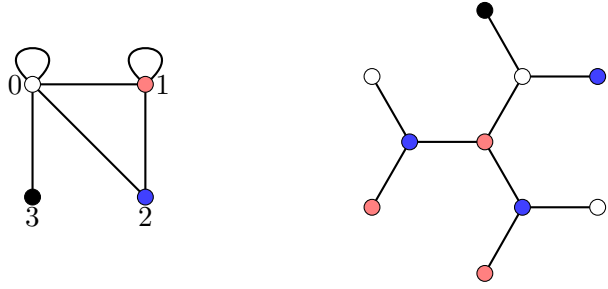
\begin{figure}[ht]
    \centering
    \captionsetup{width=.7\linewidth}
    \begin{tikzpicture}
        \coordinate[label=left:{$0$}] (v0) at (0,0);
        \coordinate[label=right:{$1$}] (v1) at (1.5,0);
        \coordinate[label=below:{$2$}] (v2) at (1.5,-1.5);
        \coordinate[label=below:{$3$}] (v3) at (0,-1.5);

        \draw[thick,black] (v0) -- (v1) -- (v2) -- (v0);
        \draw[black,thick] (v0) -- (v3);

        \draw[thick,black] (v0) to [out=40,in=140,loop,style={min distance=10mm}] (v0);
        \draw[thick,black] (v1) to [out=40,in=140,loop,style={min distance=10mm}] (v1);

        \draw[black,fill=white] (v0) circle (3pt);
        \draw[black,fill=red!50] (v1) circle (3pt);
        \draw[black,fill=blue!75] (v2) circle (3pt);
        \draw[black,fill=black] (v3) circle (3pt);
        \node at (2.5,-2.5) {};
    \end{tikzpicture}
    \hspace{1cm}
    \begin{tikzpicture}
        \coordinate (v1) at (-0.5,0.87);
        \coordinate (v2) at (-0.5,-0.87);
        \coordinate (v3) at (0,0);
        \coordinate (v4) at (1,0);
        \coordinate (v5) at (1,1.75);
        \coordinate (v6) at (1,-1.75);
        \coordinate (v7) at (1.5,0.87);
        \coordinate (v8) at (1.5,-0.87);
        \coordinate (v9) at (2.5,0.87);
        \coordinate (v10) at (2.5,-0.87);

        \draw[thick,black] (v1) -- (v3) -- (v2);
        \draw[thick,black] (v3) -- (v4) -- (v7) -- (v9);
        \draw[thick,black] (v4) -- (v8) -- (v10);
        \draw[thick,black] (v7) -- (v5);
        \draw[thick,black] (v8) -- (v6);

        \draw[black,fill=white] (v1) circle (3pt);
        \draw[black,fill=red!50] (v2) circle (3pt);
        \draw[black,fill=blue!75] (v3) circle (3pt);
        \draw[black,fill=red!50] (v4) circle (3pt);
        \draw[black,fill=black] (v5) circle (3pt);
        \draw[black,fill=red!50] (v6) circle (3pt);
        \draw[black,fill=white] (v7) circle (3pt);
        \draw[black,fill=blue!75] (v8) circle (3pt);
        \draw[black,fill=blue!75] (v9) circle (3pt);
        \draw[black,fill=white] (v10) circle (3pt);
        \node at (-1,0) {};
    \end{tikzpicture}
    \caption{The constraint graph encoding the capacity $3$ independent set model (left) and a capacity $3$ independent set on a tree (right). If spin $i$ has activity $\lambda^i$, then the configuration has weight $\lambda^{12}$.}
    \label{fig:cappacityc}
\end{figure}

The problem of maximizing or minimizing the partition function of a weighted hard-constraint spin model with constraint graph $H$ is very closely related to enumeration of $H^{\bf n}$-colorings, as we shall now see. Consider a weighted hard-constraint spin model on a graph $G$ with $H$ on vertex set $\{v_0, v_1, \ldots, v_k\}$ encoding the constraints on the spins. Let ${\boldsymbol\lambda}=(\lambda_0, \lambda_1, \ldots, \lambda_k)$ be an assignment of positive rational activities to the spins. Let integer $M$ be such that $M\lambda_i = n_i \in {\mathbb N}$ for each $i$. For any $H$-coloring $f:V(G)\rightarrow V(H)$ we have, from equation \eqref{eq:config-weight},
\[
M^nw^{(H,{\boldsymbol\lambda})}(f) = \prod_{i=0}^k n_i^{|f^{-1}(v_i)|}
\]
where $n$ is the number of vertices of $G$. Summing over all $f \in \Hom(G,H)$ and recalling the definition of the blow-up graph, we get
\begin{equation} \label{eq:weighted-partition-function}
M^n Z^{\boldsymbol\lambda}(G,H) = \hom(G, H^{\bf n}).
\end{equation}
This idea leads to the following lemma.        
\begin{lemma} \label{lem:blow-up}
Let $H$ be an arbitrary graph on vertex set $\{v_0, v_1, \ldots, v_k\}$, perhaps with loops but without multiple edges. If $G_1$ and $G_2$ are graphs on the same number of vertices with the property that $\hom(G_1, H^{\bf n}) \le \hom(G_2, H^{\bf n})$ for every vector ${\bf n}$ of positive integers of length $k+1$, then for every assignment of real positive (not necessarily rational) activities ${\boldsymbol\lambda}$ to the vertices of $H$, we have $Z^{\boldsymbol\lambda}(G_1,H) \le Z^{\boldsymbol\lambda}(G_2,H)$.   
\end{lemma} 

\begin{proof}
When the $\lambda_i$s are all rational, the result follows directly from equation \eqref{eq:weighted-partition-function}. If $(\lambda_i)_{i=0}^k$ is not a rational $(k+1)$-tuple, we can consider a sequence of rational $(k+1)$-tuples that converge to $(\lambda_i)_{i=0}^k$. The validity of \cref{lem:blow-up}, together with the continuity of $Z^{\boldsymbol\lambda}(G,H)$, allows us to conclude that $Z^{\boldsymbol\lambda}(G_1,H) \le Z^{\boldsymbol\lambda}(G_2,H)$ in this case.     
\end{proof}

The question of which graphs in various families maximize and minimize the partition function of various hard-constraint models has a well-developed literature; see e.g.~\cite{CohenPerkinsTetali2017, CutlerRadcliffe2011, Sernau2018} for the Widom-Rowlinson model, \cite{Kahn2001, Zhao2010, DaviesJenssenPerkinsRoberts2017} for the hard-core model, and \cite{CutlerKass2020, CutlerRadcliffe2011, CutlerRadcliffe2014, GalvinTetali2004, Galvin2006, SahSawhneyStonerZhao2020} for general models.

We can use \cref{lem:blow-up} to draw conclusions about the hard-core model, the Widom-Rowlinson model, and the capacity $C$ independent set model on trees.

\begin{theorem}
Let $n \ge 1$ be arbitrary and let $T_n \in \mathcal{T}_n$.
\begin{enumerate}
    \item[(a)] For the hard-core model with ${\boldsymbol\lambda}=(1,\lambda)$ for arbitrary $\lambda > 0$, we have
    \[
    Z^{\boldsymbol\lambda}(P_n,H_{\rm ind}) \le Z^{\boldsymbol\lambda}(T_n,H_{\rm ind}) \le Z^{\boldsymbol\lambda}(S_n,H_{\rm ind}).
    \]
    
    \item[(b)] For the $k$-state Widom-Rowlinson model with ${\boldsymbol\lambda}=(1,\lambda_1, \cdots, \lambda_k)$ for arbitrary $\lambda_i > 0$, we have
    \[
    Z^{\boldsymbol\lambda}(P_n,H_{{\rm WR}(k)}) \le Z^{\boldsymbol\lambda}(T_n,H_{{\rm WR}(k)}) \le Z^{\boldsymbol\lambda}(S_n,H_{{\rm WR}(k)}).
    \]
      
    \item[(c)] For the capacity $C$ independent set model with ${\boldsymbol\lambda}=(1, \lambda, \cdots, \lambda^C)$ for arbitrary $\lambda > 0$, we have
    \[
    Z^{\boldsymbol\lambda}(P_n,H_C) \le Z^{\boldsymbol\lambda}(T_n,H_C) \le Z^{\boldsymbol\lambda}(S_n,H_C).
    \]
\end{enumerate}
\end{theorem}

\begin{proof}
All three upper bounds follow from \cref{thm:siderenko} and \cref{lem:blow-up}. For the lower bounds, it suffices to establish that arbitrary positive integer blow-ups of $H_{\rm ind}$, $H_{{\rm WR}(k)}$, and $H_C$ are Hoffman-London, and then appeal to \cref{lem:blow-up}. For $H_{\rm ind}$ and $H_C$, we use that arbitrary blow-ups of loop threshold graphs are still loop threshold and appeal to \cref{thm:loop_threshold}, and we use \cref{cor:blowup-of-fully-looped-star} for $H_{{\rm WR}(k)}$. 
\end{proof}

\section{Classification of target graphs with at most three vertices} \label{sec:constraint-3-vertices}

There are twenty-eight graphs (with or without loops, not necessarily connected) on three or fewer vertices. These are shown in \cref{fig:H-on-3-or-fewer-vertices}. For each $H$ on this list, each $n$, and each $T_n \in {\mathcal T}_n$, it holds that $\hom(P_n, H) \le \hom(T_n,H)$. The collection of results we have presented thus far establish this fact, and also allow us, for each $H$, to completely specify which other $T_n \in {\mathcal T}_n$, if any, also minimize the count of $H$-colorings. As $\hom(T_1,H)=|V(H)|$ for every $H$, we confine attention to trees on $n \ge 2$ vertices.

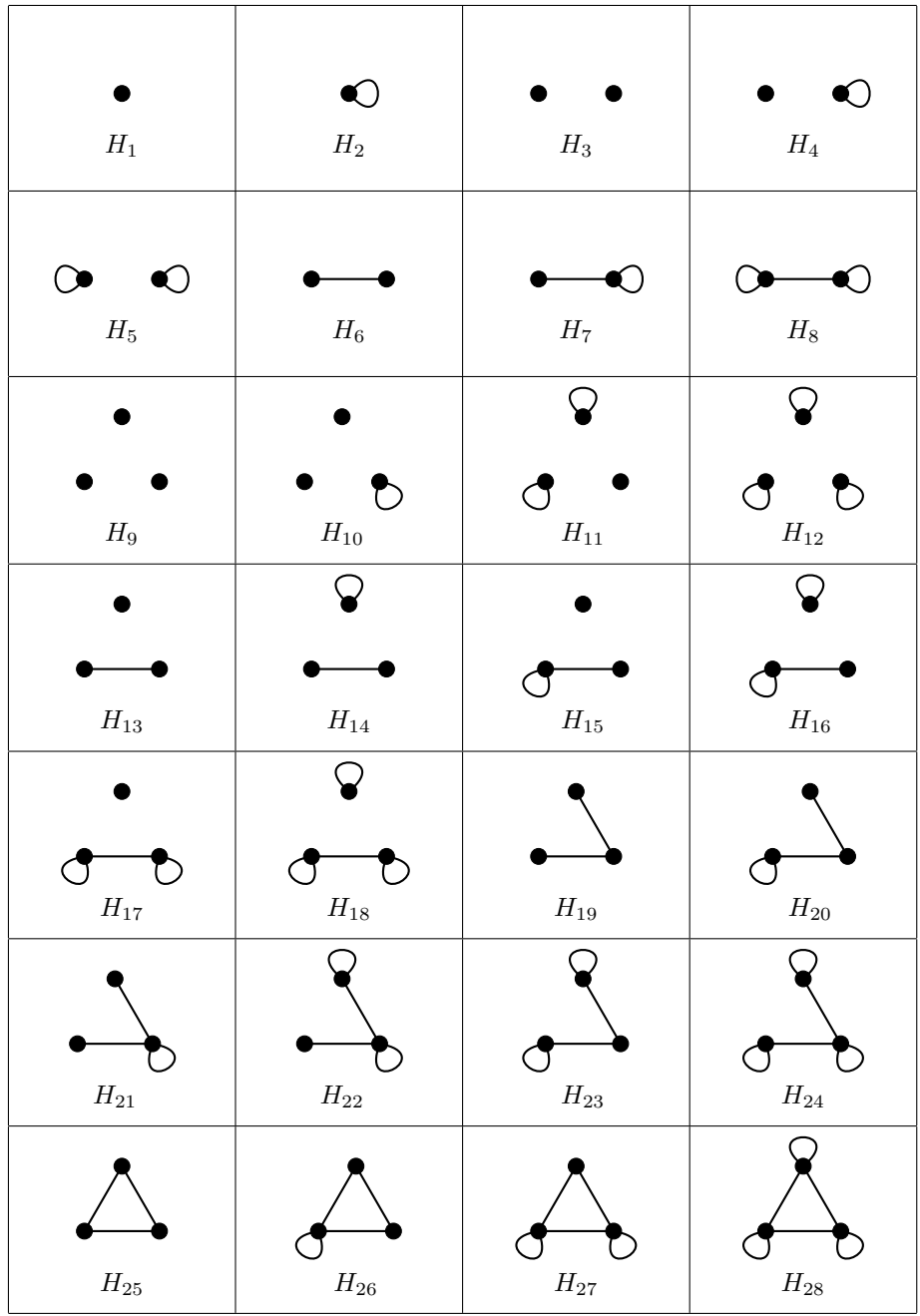
\begin{figure}[ht!]
\begin{center}
{\small
\begin{tabular}{ | c | c | c | c |}
\hline
\begin{tikzpicture}
\coordinate (A) at (0, 0);
\draw[fill=black] (A) circle (3pt);
\node[] at (0,-0.7) {$H_1$};
\node[] at (-1,0) {};
\node[] at (1,0) {};
\node[] at  (0, 1.05) {};
\node[] at (0,-1.05) {};
\end{tikzpicture}
&
\begin{tikzpicture}
\coordinate (A) at (0, 0);
\draw[fill=black] (A) circle (3pt);
\draw[thick] (A) to [out=-50,in=50,loop,style={min distance=8mm}] (A);
\node[] at (0,-0.7) {$H_2$};
\node[] at (-1,0) {};
\node[] at (1,0) {};
\node[] at  (0, 1.05) {};
\node[] at (0,-1.05) {};
\end{tikzpicture}
&
\begin{tikzpicture}
\coordinate (A) at (-0.5, 0);
\coordinate (B) at  (0.5, 0);
\draw[fill=black] (A) circle (3pt);
\draw[fill=black] (B) circle (3pt);
\node[] at (0,-0.7) {$H_3$};
\node[] at (-1,0) {};
\node[] at (1,0) {};
\node[] at  (0, 1.05) {};
\node[] at (0,-1.05) {};
\end{tikzpicture}
&
\begin{tikzpicture}
\coordinate (A) at (-0.5, 0);
\coordinate (B) at  (0.5, 0);
\draw[fill=black] (A) circle (3pt);
\draw[fill=black] (B) circle (3pt);
\draw[thick] (B) to [out=-50,in=50,loop,style={min distance=8mm}] (B);
\node[] at (0,-0.7) {$H_4$};
\node[] at (-1,0) {};
\node[] at (1,0) {};
\node[] at  (0, 1.05) {};
\node[] at (0,-1.05) {};
\end{tikzpicture}
\\
\hline
\begin{tikzpicture}
\coordinate (A) at (-0.5, 0);
\coordinate (B) at  (0.5, 0);
\draw[fill=black] (A) circle (3pt);
\draw[fill=black] (B) circle (3pt);
\draw[thick] (A) to [out=230,in=130,loop,style={min distance=8mm}] (A);
\draw[thick] (B) to [out=-50,in=50,loop,style={min distance=8mm}] (B);
\node[] at (0,-0.7) {$H_5$};
\node[] at (-1,0) {};
\node[] at (1,0) {};
\node[] at  (0, 1.05) {};
\node[] at (0,-1.05) {};
\end{tikzpicture}
&

\begin{tikzpicture}
\coordinate (A) at (-0.5, 0);
\coordinate (B) at  (0.5, 0);
\draw[fill=black] (A) circle (3pt);
\draw[fill=black] (B) circle (3pt);
\draw[black, thick] (A)--(B);
\node[] at (0,-0.7) {$H_6$};
\node[] at (-1,0) {};
\node[] at (1,0) {};
\node[] at  (0, 1.05) {};
\node[] at (0,-1.05) {};
\end{tikzpicture}
&
\begin{tikzpicture}
\coordinate (A) at (-0.5, 0);
\coordinate (B) at  (0.5, 0);
\coordinate (C) at  (0, 0.866);
\draw[fill=black] (A) circle (3pt);
\draw[fill=black] (B) circle (3pt);
\draw[black, thick] (A)--(B);
\draw[thick] (B) to [out=-50,in=50,loop,style={min distance=8mm}] (B);
\node[] at (0,-0.7) {$H_7$};
\node[] at (-1,0) {};
\node[] at (1,0) {};
\node[] at  (0, 1.05) {};
\node[] at (0,-1.05) {};
\end{tikzpicture}
&
\begin{tikzpicture}
\coordinate (A) at (-0.5, 0);
\coordinate (B) at  (0.5, 0);
\draw[fill=black] (A) circle (3pt);
\draw[fill=black] (B) circle (3pt);
\draw[black, thick] (A)--(B);
\draw[thick] (A) to [out=230,in=130,loop,style={min distance=8mm}] (A);
\draw[thick] (B) to [out=-50,in=50,loop,style={min distance=8mm}] (B);
\node[] at (0,-0.7) {$H_8$};
\node[] at (-1,0) {};
\node[] at (1,0) {};
\node[] at  (0, 1.05) {};
\node[] at (0,-1.05) {};
\end{tikzpicture}
\\
\hline
\begin{tikzpicture}
\coordinate (A) at (-0.5, 0);
\coordinate (B) at  (0.5, 0);
\coordinate (C) at  (0, 0.866);
\draw[fill=black] (A) circle (3pt);
\draw[fill=black] (B) circle (3pt);
\draw[fill=black] (C) circle (3pt);
\node[] at (0,-0.7) {$H_9$};
\node[] at (-1,0) {};
\node[] at (1,0) {};
\end{tikzpicture}
&
\begin{tikzpicture}
\coordinate (A) at (-0.5, 0);
\coordinate (B) at  (0.5, 0);
\coordinate (C) at  (0, 0.866);
\draw[fill=black] (A) circle (3pt);
\draw[fill=black] (B) circle (3pt);
\draw[fill=black] (C) circle (3pt);
\draw[thick] (B) to [out=250,in=350,loop,style={min distance=8mm}] (B);
\node[] at (0,-0.7) {$H_{10}$};
\node[] at (-1,0) {};
\node[] at (1,0) {};
\end{tikzpicture}
&
\begin{tikzpicture}
\coordinate (A) at (-0.5, 0);
\coordinate (B) at  (0.5, 0);
\coordinate (C) at  (0, 0.866);
\draw[fill=black] (A) circle (3pt);
\draw[fill=black] (B) circle (3pt);
\draw[fill=black] (C) circle (3pt);
\draw[thick] (A) to [out=190,in=290,loop,style={min distance=8mm}] (A);
\draw[thick] (C) to [out=40,in=140,loop,style={min distance=8mm}] (C);
\node[] at (0,-0.7) {$H_{11}$};
\node[] at (-1,0) {};
\node[] at (1,0) {};
\end{tikzpicture}
&
\begin{tikzpicture}
\coordinate (A) at (-0.5, 0);
\coordinate (B) at  (0.5, 0);
\coordinate (C) at  (0, 0.866);
\draw[fill=black] (A) circle (3pt);
\draw[fill=black] (B) circle (3pt);
\draw[fill=black] (C) circle (3pt);
\draw[thick] (A) to [out=190,in=290,loop,style={min distance=8mm}] (A);
\draw[thick] (B) to [out=250,in=350,loop,style={min distance=8mm}] (B);
\draw[thick] (C) to [out=40,in=140,loop,style={min distance=8mm}] (C);
\node[] at (0,-0.7) {$H_{12}$};
\node[] at (-1,0) {};
\node[] at (1,0) {};
\end{tikzpicture}
\\
\hline
\begin{tikzpicture}
\coordinate (A) at (-0.5, 0);
\coordinate (B) at  (0.5, 0);
\coordinate (C) at  (0, 0.866);
\draw[fill=black] (A) circle (3pt);
\draw[fill=black] (B) circle (3pt);
\draw[fill=black] (C) circle (3pt);
\draw[black, thick] (A)--(B);
\node[] at (0,-0.7) {$H_{13}$};
\node[] at (-1,0) {};
\node[] at (1,0) {};
\end{tikzpicture}
&
\begin{tikzpicture}
\coordinate (A) at (-0.5, 0);
\coordinate (B) at  (0.5, 0);
\coordinate (C) at  (0, 0.866);
\draw[fill=black] (A) circle (3pt);
\draw[fill=black] (B) circle (3pt);
\draw[fill=black] (C) circle (3pt);
\draw[black, thick] (A)--(B);
\draw[thick] (C) to [out=40,in=140,loop,style={min distance=8mm}] (C);
\node[] at (0,-0.7) {$H_{14}$};
\node[] at (-1,0) {};
\node[] at (1,0) {};
\end{tikzpicture}
&
\begin{tikzpicture}
\coordinate (A) at (-0.5, 0);
\coordinate (B) at  (0.5, 0);
\coordinate (C) at  (0, 0.866);
\draw[fill=black] (A) circle (3pt);
\draw[fill=black] (B) circle (3pt);
\draw[fill=black] (C) circle (3pt);
\draw[black, thick] (A)--(B);
\draw[thick] (A) to [out=190,in=290,loop,style={min distance=8mm}] (A);
\node[] at (0,-0.7) {$H_{15}$};
\node[] at (-1,0) {};
\node[] at (1,0) {};
\end{tikzpicture}
&
\begin{tikzpicture}
\coordinate (A) at (-0.5, 0);
\coordinate (B) at  (0.5, 0);
\coordinate (C) at  (0, 0.866);
\draw[fill=black] (A) circle (3pt);
\draw[fill=black] (B) circle (3pt);
\draw[fill=black] (C) circle (3pt);
\draw[black, thick] (A)--(B);
\draw[thick] (A) to [out=190,in=290,loop,style={min distance=8mm}] (A);
\draw[thick] (C) to [out=40,in=140,loop,style={min distance=8mm}] (C);
\node[] at (0,-0.7) {$H_{16}$};
\node[] at (-1,0) {};
\node[] at (1,0) {};
\end{tikzpicture}
\\
\hline
\begin{tikzpicture}
\coordinate (A) at (-0.5, 0);
\coordinate (B) at  (0.5, 0);
\coordinate (C) at  (0, 0.866);
\draw[fill=black] (A) circle (3pt);
\draw[fill=black] (B) circle (3pt);
\draw[fill=black] (C) circle (3pt);
\draw[black, thick] (A)--(B);
\draw[thick] (A) to [out=190,in=290,loop,style={min distance=8mm}] (A);
\draw[thick] (B) to [out=250,in=350,loop,style={min distance=8mm}] (B);
\node[] at (0,-0.7) {$H_{17}$};
\node[] at (-1,0) {};
\node[] at (1,0) {};
\end{tikzpicture}
&
\begin{tikzpicture}
\coordinate (A) at (-0.5, 0);
\coordinate (B) at  (0.5, 0);
\coordinate (C) at  (0, 0.866);
\draw[fill=black] (A) circle (3pt);
\draw[fill=black] (B) circle (3pt);
\draw[fill=black] (C) circle (3pt);
\draw[black, thick] (A)--(B);
\draw[thick] (A) to [out=190,in=290,loop,style={min distance=8mm}] (A);
\draw[thick] (B) to [out=250,in=350,loop,style={min distance=8mm}] (B);
\draw[thick] (C) to [out=40,in=140,loop,style={min distance=8mm}] (C);
\node[] at (0,-0.7) {$H_{18}$};
\node[] at (-1,0) {};
\node[] at (1,0) {};
\end{tikzpicture}
&
\begin{tikzpicture}
\coordinate (A) at (-0.5, 0);
\coordinate (B) at  (0.5, 0);
\coordinate (C) at  (0, 0.866);
\draw[fill=black] (A) circle (3pt);
\draw[fill=black] (B) circle (3pt);
\draw[fill=black] (C) circle (3pt);
\draw[black, thick] (A)--(B);
\draw[black, thick] (B)--(C);
\node[] at (0,-0.7) {$H_{19}$};
\node[] at (-1,0) {};
\node[] at (1,0) {};
\end{tikzpicture}
&
\begin{tikzpicture}
\coordinate (A) at (-0.5, 0);
\coordinate (B) at  (0.5, 0);
\coordinate (C) at  (0, 0.866);
\draw[fill=black] (A) circle (3pt);
\draw[fill=black] (B) circle (3pt);
\draw[fill=black] (C) circle (3pt);
\draw[black, thick] (A)--(B);
\draw[black, thick] (B)--(C);
\draw[thick] (A) to [out=190,in=290,loop,style={min distance=8mm}] (A);
\node[] at (0,-0.7) {$H_{20}$};
\node[] at (-1,0) {};
\node[] at (1,0) {};
\end{tikzpicture}
\\
\hline
\begin{tikzpicture}
\coordinate (A) at (-0.5, 0);
\coordinate (B) at  (0.5, 0);
\coordinate (C) at  (0, 0.866);
\draw[fill=black] (A) circle (3pt);
\draw[fill=black] (B) circle (3pt);
\draw[fill=black] (C) circle (3pt);
\draw[black, thick] (A)--(B);
\draw[black, thick] (B)--(C);
\draw[thick] (B) to [out=250,in=350,loop,style={min distance=8mm}] (B);
\node[] at (0,-0.7) {$H_{21}$};
\node[] at (-1,0) {};
\node[] at (1,0) {};
\end{tikzpicture}
&
\begin{tikzpicture}
\coordinate (A) at (-0.5, 0);
\coordinate (B) at  (0.5, 0);
\coordinate (C) at  (0, 0.866);
\draw[fill=black] (A) circle (3pt);
\draw[fill=black] (B) circle (3pt);
\draw[fill=black] (C) circle (3pt);
\draw[black, thick] (A)--(B);
\draw[black, thick] (B)--(C);
\draw[thick] (B) to [out=250,in=350,loop,style={min distance=8mm}] (B);
\draw[thick] (C) to [out=40,in=140,loop,style={min distance=8mm}] (C);
\node[] at (0,-0.7) {$H_{22}$};
\node[] at (-1,0) {};
\node[] at (1,0) {};
\end{tikzpicture}
&
\begin{tikzpicture}
\coordinate (A) at (-0.5, 0);
\coordinate (B) at  (0.5, 0);
\coordinate (C) at  (0, 0.866);
\draw[fill=black] (A) circle (3pt);
\draw[fill=black] (B) circle (3pt);
\draw[fill=black] (C) circle (3pt);
\draw[black, thick] (A)--(B);
\draw[black, thick] (B)--(C);
\draw[thick] (A) to [out=190,in=290,loop,style={min distance=8mm}] (A);
\draw[thick] (C) to [out=40,in=140,loop,style={min distance=8mm}] (C);
\node[] at (0,-0.7) {$H_{23}$};
\node[] at (-1,0) {};
\node[] at (1,0) {};
\end{tikzpicture}
&
\begin{tikzpicture}
\coordinate (A) at (-0.5, 0);
\coordinate (B) at  (0.5, 0);
\coordinate (C) at  (0, 0.866);
\draw[fill=black] (A) circle (3pt);
\draw[fill=black] (B) circle (3pt);
\draw[fill=black] (C) circle (3pt);
\draw[black, thick] (A)--(B);
\draw[black, thick] (B)--(C);
\draw[thick] (A) to [out=190,in=290,loop,style={min distance=8mm}] (A);
\draw[thick] (B) to [out=250,in=350,loop,style={min distance=8mm}] (B);
\draw[thick] (C) to [out=40,in=140,loop,style={min distance=8mm}] (C);
\node[] at (0,-0.7) {$H_{24}$};
\node[] at (-1,0) {};
\node[] at (1,0) {};
\end{tikzpicture}
\\
\hline
\begin{tikzpicture}
\coordinate (A) at (-0.5, 0);
\coordinate (B) at  (0.5, 0);
\coordinate (C) at  (0, 0.866);
\draw[fill=black] (A) circle (3pt);
\draw[fill=black] (B) circle (3pt);
\draw[fill=black] (C) circle (3pt);
\draw[black, thick] (A)--(B);
\draw[black, thick] (B)--(C);
\draw[black, thick] (A)--(C);
\node[] at (0,-0.7) {$H_{25}$};
\node[] at (-1,0) {};
\node[] at (1,0) {};
\end{tikzpicture}
&
\begin{tikzpicture}
\coordinate (A) at (-0.5, 0);
\coordinate (B) at  (0.5, 0);
\coordinate (C) at  (0, 0.866);
\draw[fill=black] (A) circle (3pt);
\draw[fill=black] (B) circle (3pt);
\draw[fill=black] (C) circle (3pt);
\draw[black, thick] (A)--(B);
\draw[black, thick] (B)--(C);
\draw[black, thick] (A)--(C);
\draw[thick] (A) to [out=190,in=290,loop,style={min distance=8mm}] (A);
\node[] at (0,-0.7) {$H_{26}$};
\node[] at (-1,0) {};
\node[] at (1,0) {};
\end{tikzpicture}
&
\begin{tikzpicture}
\coordinate (A) at (-0.5, 0);
\coordinate (B) at  (0.5, 0);
\coordinate (C) at  (0, 0.866);
\draw[fill=black] (A) circle (3pt);
\draw[fill=black] (B) circle (3pt);
\draw[fill=black] (C) circle (3pt);
\draw[black, thick] (A)--(B);
\draw[black, thick] (B)--(C);
\draw[black, thick] (A)--(C);
\draw[thick] (A) to [out=190,in=290,loop,style={min distance=8mm}] (A);
\draw[thick] (B) to [out=250,in=350,loop,style={min distance=8mm}] (B);
\node[] at (0,-0.7) {$H_{27}$};
\end{tikzpicture}
&
\begin{tikzpicture}
\coordinate (A) at (-0.5, 0);
\coordinate (B) at  (0.5, 0);
\coordinate (C) at  (0, 0.866);
\draw[fill=black] (A) circle (3pt);
\draw[fill=black] (B) circle (3pt);
\draw[fill=black] (C) circle (3pt);
\draw[black, thick] (A)--(B);
\draw[black, thick] (B)--(C);
\draw[black, thick] (A)--(C);
\draw[thick] (A) to [out=190,in=290,loop,style={min distance=8mm}] (A);
\draw[thick] (B) to [out=250,in=350,loop,style={min distance=8mm}] (B);
\draw[thick] (C) to [out=40,in=140,loop,style={min distance=8mm}] (C);
\node[] at (0,-0.7) {$H_{28}$};
\node[] at (-1,0) {};
\node[] at (1,0) {};
\end{tikzpicture}
\\
\hline
\end{tabular}
}
\caption{All target graphs on up to three vertices.}
\label{fig:H-on-3-or-fewer-vertices}
\end{center}
\end{figure}

We begin with those cases in which each tree on $n \ge 2$ vertices admits the same number of $H$-colorings, using \cref{obs:isolated} without additional mention to discuss sets of target graphs that behave identically. For unions of isolated vertices, namely $H_1$, $H_3$, and $H_9$, there are no $H$-colorings of any such tree. We apply \cref{obs:regular} to those $H$ whose non-isolated components are regular to find our trees have unique $H_2$-, $H_4$-, and $H_{10}$-colorings, two $H_6$- and $H_{13}$-colorings, $2^n$ $H_8$- and $H_{17}$-colorings, $3^n$ $H_{28}$-colorings, and finally $3\cdot2^{n-1}$ $H_{23}$- and $H_{25}$-colorings. Then we use \cref{obs:union_of_comp} to consider target graphs with regular components to see each of these trees have two $H_5$-colorings, two $H_{11}$-colorings, three $H_{12}$-colorings, three $H_{14}$-colorings, and $2^n+1$ $H_{18}$-colorings.

Because $H_{19}$ is a complete bipartite graph with unequal parts, \cref{thm:complete_bipartite} tells us the trees minimizing $\hom(T_n,H_{19})$ are precisely those with a balanced bipartition.

\cref{thm:dom_reg} demonstrates that regular graphs to which looped dominating vertices are added are strongly Hoffman-London, including $H_7$, $H_{21}$, $H_{24}$, and $H_{26}$. Furthermore, $H_{15}$ is strongly Hoffman-London due to \cref{thm:dom_reg} and \cref{obs:isolated} and $H_{16}$ is strongly Hoffman-London due to \cref{thm:dom_reg} and \cref{obs:union_of_comp}. 

Note that $H_{20} \times K_2 = P_6$. That $\hom(P_n,H_{20}) \le \hom(T_n,H_{20})$ for all $n$ and $T_n \in {\mathcal T_n}$ now follows from \cref{obs:cartesian_2} and \cref{thm:paths_and_stars_as_H}, and that $\hom(P_n,H_{20}) < \hom(T_n,H_{20})$ for all $n$ and $T_n \in {\mathcal T_n}$ different from $P_n$ follows from \cref{thm:even_path_minimizer}. Finally, $H_{22}$ and $H_{27}$ are loop threshold and thus strongly Hoffman-London by \cref{thm:loop_threshold}.

For convenience, we summarize these results in \cref{tab:28}. In cases where all trees on $n$ vertices admit the same number of $H$-colorings, that number is given.

\begin{table}[ht!]
\begin{center}
\begin{tabular}{|c|c||c|c||c|c|}
\hline
Identifier & Minimizer & Identifier & Minimizer& Identifier & Minimizer\\
\hline
$H_1$ & All trees ($0$) & $H_{11}$ & All trees ($2$) & $H_{20}$ & Paths \\
$H_2$ & All trees ($1$) & $H_{12}$ & All trees ($3$) & $H_{21}$ & Paths \\
$H_3$ & All trees ($0$) & $H_{13}$ & All trees ($2$) & $H_{22}$ & Paths\\
$H_4$ & All trees ($1$) & $H_{14}$ & All trees ($3$) & $H_{23}$ & All trees ($3\cdot2^{n-1}$)\\
$H_5$ & All trees ($2$) & $H_{15}$ & Paths & $H_{24}$ & Paths\\
$H_6$ & All trees ($2$) & $H_{16}$ & Paths & $H_{25}$ & All trees ($3\cdot2^{n-1}$)\\
$H_7$ & Paths & $H_{17}$ & All trees ($2^n$) & $H_{26}$ & Paths \\
$H_8$ & All trees ($2^n$) & $H_{18}$ & All trees ($2^n + 1$) & $H_{27}$ & Paths\\
$H_9$ & All trees ($0$) & $H_{19}$ & Trees with balanced & $H_{28}$&All trees ($3^n$)\\
$H_{10}$ & All trees ($1$) &  &  bipartitions  &&\\
\hline
\end{tabular}
\caption{Classification of tree minimizers of $\hom(T_n,H)$ for $n \ge 2$.}
\label{tab:28}
\end{center}
\end{table}

\section{Open problems} \label{sec:open_problems}

We conclude with a few questions that remain open and which we find of interest.

In addition to \cref{prob:CL}, Csikv\'{a}ri and Lin propose a weaker variant in which ``all $n$'' is replaced by ``all sufficiently large $n$.'' To consider this version of the problem, we introduce the following terminology.

\begin{definition}
We say a target graph $H$ is \emph{eventually Hoffman-London} if for all sufficiently large $n$,
\[ \hom(P_n,H) = \min_{T_n \in \mathcal{T}_n} \hom(T_n,H). \]
If it holds that $\hom(P_n,H) < \hom(T_n,H)$ for all sufficiently large $n$ and $T_n \in \mathcal{T}_n\setminus \{P_n\}$, we say that $H$ \emph{eventually strongly Hoffman-London}. 
\end{definition}

At present, no target graphs are known to be eventually Hoffman-London but not Hoffman-London (nor eventually strongly Hoffman-London but not strongly Hoffman-London). In other words, there are no known target graphs $H$ or which there is a non-path tree $T_n$ admitting fewer $H$-colorings than $P_n$ for a specific $n$ but for which the path minimizes $\hom(T_n,H)$ when $n$ is sufficiently large. To this end, we introduce the following problem.

\begin{problem}
    Does there exist a target graph $H$ and integers $\alpha < \beta$ such that
    \[ \hom(P_{\alpha},H) > \min_{T_{\alpha} \in \mathcal{T}_{\alpha}} \hom(T_{\alpha},H), \]
    but
    \[ \hom(P_{n},H) = \min_{T_n \in \mathcal{T}_n} \hom(T_n,H) \]
    for each $n \ge \beta$?
\end{problem}

\cref{thm:even_path_minimizer} proves that paths on an even number of vertices other than $P_2$ are strongly Hoffman-London. As mentioned, \cref{thm:complete_bipartite} demonstrates that the case is more complicated for $P_3$ as the set of trees with minimal $P_3$-colorings is neither just the path nor all trees. For paths with a larger odd number of vertices, the question remains open. 

\begin{problem}
    For $m \ge 2$, classify the $T \in \mathcal{T}_n$ that minimize $\hom(T,P_{2m+1})$.
\end{problem}

\cref{thm:complete_bipartite} shows that $K_{a,b}$ is Hoffman-London, but it remains open whether complete multipartite graphs $H$ with more than two parts need be Hoffman-London. 

\begin{problem}
    Let $r \ge 3$ and $H$ be a complete $r$-partite graph. Does it follow that $H$ is Hoffman-London?
\end{problem}

Note that $K_{a_1, \ldots, a_q}$-colorings may be viewed as weighted $q$-colorings, so this question is equivalent to asking whether there exist weights such that some tree on $n$ vertices has fewer weighted $q$-colorings than $P_n$.

As discussed in \cref{ssec:two_auto_classes}, we proved that each graph in the family of $H(a,b,\ell)$, in which each vertex of a $b$-clique is replaced with a bouquet of $\ell$ $a$-cliques, is Hoffman-London except in the case $a \ge 3, b = 1$, and $\ell \ge 2$.

\begin{problem}
    For which values of $a \ge 3$ and $\ell \ge 2$, if any, is $H(a,1,\ell)$ Hoffman-London?
\end{problem}

Finally, we proved in \cref{thm:loop_threshold} that loop threshold graphs are Hoffman-London. There are three closely related classes of graphs for which it would be interesting to prove an analogous result:
\begin{enumerate}
\item Threshold graphs: the graphs in the smallest family containing $K_1$ that is closed under adding isolated vertices and (unlooped) dominating vertices.
\item Quasi-loop threshold graphs: the graphs in the smallest family containing $K_1$ and $K_1^\loop$ that is closed under adding isolated vertices, adding looped dominating vertices, and taking disjoint unions.
\item Quasi-threshold graphs: the graphs in the smallest family containing $K_1$ that is closed under adding isolated vertices, adding (unlooped) dominating vertices, and taking disjoint unions.
\end{enumerate}
There are threshold (and so quasi-threshold) graphs on four vertices that the methods and results discussed in this paper cannot handle, such as the triangle with a pendant edge. There are also quasi-loop threshold graphs on four vertices whose status remains open, including the graph obtained from the union of a loop and two isolated vertices by adding a dominating vertex.

\section*{Acknowledgments}

This work was started at the workshop ``Graph Theory: structural properties, labelings, and connections to applications'' hosted by the American Institute of Mathematics (AIM), Pasedena CA, July 22--July 26, 2024. The authors thank AIM and the organizers of the workshop for facilitating their collaboration.

\end{document}